\documentclass[12pt]{article}
\usepackage{amssymb,amsmath,amsfonts,mathrsfs,bm,galois,graphicx}
\usepackage{longtable}\usepackage{booktabs}\usepackage{threeparttable}
\usepackage{natbib}
\usepackage{subcaption}
\usepackage{dutchcal}
\usepackage{multirow}
\usepackage{makecell}
\usepackage{authblk}
\usepackage{hyperref}
\usepackage{amsthm}
\usepackage{xcolor}

\usepackage{float}     
\usepackage{caption}   

\usepackage{bm}         
\usepackage{booktabs}   
\usepackage{graphicx}

\hypersetup{
    colorlinks=true,
    linkcolor=red,
    citecolor=blue,
    filecolor=magenta,
    urlcolor=cyan
}
\newtheorem{theorem}{Theorem}
\newtheorem{lemma}{Lemma}

\newtheorem{assumption}{Assumption}

\newtheorem{remark}{Remark}

\numberwithin{equation}{section}
\def\T{{ \mathrm{\scriptscriptstyle T} }}
\makeatletter

\newcommand{\Rmnum}[1]{\expandafter\@slowromancap\romannumeral #1@}

\newcommand{\E}{\operatorname{E}}           
\newcommand{\PP}{\operatorname{P}}        
\newcommand{\var}{\operatorname{Var}} 
\newcommand{\cov}{\operatorname{Cov}} 

\makeatother

\begin{document}
\baselineskip=18pt

\title{Bootstrap Consistency for Empirical Likelihood in Density Ratio Models}

\author[a]{Weiwei Zhuang}
\author[b]{Weiqi Yang}
\author[c]{Jiahua Chen\thanks{Corresponding author: Jiahua Chen \quad  Email: jhchen@stat.ubc.ca}}
\affil[a,b]{School of Management, University of Science and Technology of China, Hefei, Anhui, China}
\affil[c]{Department of Statistics, The University of British Columbia, Vancouver, Canada}
\date{}

\maketitle

\begin{abstract}
We establish the validity of bootstrap methods for empirical likelihood (EL) inference under the density ratio model (DRM). In particular, we prove that the bootstrap maximum EL estimators share the same limiting distribution as their population counterparts, both at the parameter level and for distribution functionals. Our results extend existing pointwise convergence theory to weak convergence of processes, which in turn justifies bootstrap inference for quantiles and dominance indices within the DRM framework. These theoretical guarantees close an important gap in the literature, providing rigorous foundations for resampling-based confidence intervals and hypothesis tests. Simulation studies further demonstrate the accuracy and practical value of the proposed approach.

\medskip

\medskip

\noindent \textbf{Mathematics Subject Classifications (2020)~}: 62G09, 62E20, 62G10, 62G30
\end{abstract}

\baselineskip=18pt

\section{Introduction}

The empirical likelihood (EL) provides a powerful framework for statistical inference without restrictive parametric assumptions. Since its inception \citep{Owen1988}, EL has been applied in linear models \citep{Owen1991}, estimating equations \citep{94qin}, and sample surveys \citep{chen1993,wu2001}. Reviews in economics \citep{Parente}, survey methodology \citep{wu2016}, and the comprehensive account in \cite{17qin} attest to its broad and continuing impact. 

A particularly important setting for EL inference is the \emph{density ratio model} (DRM), introduced by \cite{anderson1979multivariate} and further developed by \cite{qin1997goodness}. The DRM has been widely studied \citep{98qin,fokianos2001semiparametric,chen2013,cai2017hypothesis,zhang2022density} and applied to time series \citep{kedem2008forecasting}, mixture models \citep{tan2009note,li2017semiparametric}, regression \citep{huang2012proportional}, quality monitoring \citep{hu2018r}, and dominance indices \citep{zhuang}. Compared with traditional approaches, EL methods under DRM are statistically efficient by combining information across multiple samples.

Maximum EL estimators under DRM are asymptotically normal with well-defined variance structures. Yet their variances depend on unknown features and are analytically complex, limiting practical use. This difficulty has motivated widespread application of the bootstrap. Although simulations suggest accurate finite-sample performance, theoretical guarantees have been incomplete.

This paper provides the required theoretical validation. We prove that bootstrap estimators of both parameters and distribution functionals share the same limiting laws as their population counterparts. Our results extend pointwise convergence to weak process convergence, thereby justifying bootstrap inference for quantiles and dominance indices. These findings establish a rigorous foundation for bootstrap-based confidence intervals and hypothesis tests under DRM.

The remainder of the paper is organized as follows. Section~\ref{se2} reviews the DRM framework. Section~\ref{se3} presents bootstrap consistency theorems. Section~\ref{se4} develops applications to quantiles and dominance indices. Section~\ref{se5} reports simulations, and Section~\ref{se6} illustrates practical performance.

\section{Preliminary}
\label{se2}

Suppose $\{x_{kj}, j=1,\ldots,n_k\}$ are independent and identically distributed (i.i.d.)
observations from $F_k$, $k=0,1,\ldots,m$.
Let $n=\sum_{k=0}^m n_k$ be the total sample size and $\rho_k = n_k/n$, $k=0,1,\ldots,m$.
We assume these distributions are related through the density ratio model (DRM)
\begin{equation}
\label{drm}
dF_k(x) = \exp\{{\bm\theta}_k^\T {\bf q}(x)\}\, dF_0(x), \quad k=1,\ldots,m.
\end{equation}
Here ${\bf q}(\cdot)$ is a known vector-valued basis function and
${\bm\theta}^\T =({\bm\theta}_1^\T,{\bm\theta}_2^\T,\ldots,{\bm\theta}_m^\T)$
is an unknown vector of parameters.
Without loss of generality, we set ${\bm\theta}_0={\bf 0}$.
Under the DRM, all distributions share the same support,
a natural assumption in relevant applications.

\citet{qin1997goodness} and \citet{98qin} appear to be the first to introduce EL
to data analysis under the DRM.
Let $p_{kj} = dF_0(x_{kj})$. Then the empirical likelihood under DRM is
\[
L_n(F_0,\ldots,F_m)
= \prod_{k,j} dF_k(x_{kj})
= \prod_{k,j} p_{kj} \times
\exp\left\{\sum_{k,j} {\bm\theta}_k^\T {\bf q}(x_{kj})\right\},
\]
where the products and sums range over $k=0,\ldots,m$ and $j=1,\ldots,n_k$.
The function $L_n$ depends on ${\bm\theta}$ and $F_0$, and its logarithm is
\[
l_n({\bm\theta}, F_0)
= \sum_{k,j}\log(p_{kj})
+ \sum_{k,j} {\bm\theta}_k^\T {\bf q}(x_{kj}).
\]

Inference on ${\bm\theta}$ is typically derived by profiling the EL with respect to $F_0$.
Following \cite{chen2013}, the profile log-EL is
\[
\tilde{l}_n({\bm\theta})
= -\sum_{k,j}
\log\left\{1+\sum_{s=1}^m \nu_s \big[\exp({\bm\theta}_s^\T {\bf q}(x_{kj}))-1\big]\right\}
+ \sum_{k,j} {\bm\theta}_k^\T {\bf q}(x_{kj}),
\]
where the vector $\bm{\nu} = (\nu_1,\ldots,\nu_m)^\T$ is determined by solving the 
constraint system for the implied probabilities
\[
p_{kj} = n^{-1} \left\{ 1 + \sum_{s=1}^m \nu_s
\big[\exp({\bm\theta}_s^\T {\bf q}(x_{kj}))-1\big] \right\}^{-1},
\]
subject to the $m$ constraints
\[
\sum_{k,j}p_{kj}\exp\{{\bm\theta}_r^\T {\bf q}(x_{kj})\}=1,
\quad r=1,\ldots,m.
\]

It is straightforward to show that $\tilde{l}_n$ achieves the same maximum value,
and at the same maximizer, as the simpler dual log-EL:
\begin{equation}
\label{dual_like}
l_n({\bm\theta})
= -\sum_{k,j}\log\Bigg[\sum_{s=0}^m\rho_s
\exp\{{\bm\theta}_s^\T {\bf q}(x_{kj})\}\Bigg]
+ \sum_{k,j} {\bm\theta}_k^\T {\bf q}(x_{kj}),
\end{equation}
where $\rho_s=n_s/n$.
The dual log-EL enjoys many of the same analytical properties as an ordinary parametric likelihood. 
It often leads to asymptotically normal maximum EL estimators and likelihood ratio statistics with 
chi-squared limits, providing a convenient foundation for inference under the DRM.

The maximum empirical likelihood estimator (MELE) of ${\bm\theta}$ is the maximizer 
of \eqref{dual_like}, denoted $\widehat{{\bm \theta}}$.
Given $\widehat{{\bm \theta}}$, the fitted values of $p_{kj}$ are
\[
\widehat{p}_{kj} = \big\{n\cdot h(x_{kj};\widehat{{\bm\theta}})\big\}^{-1},
\quad
h(x;{\bm\theta})=\sum_{k=0}^m \rho_k \exp({\bm\theta}_k^\T {\bf q}(x)).
\]
Consequently, the MELE of $F_r(x)$ is
\begin{align*}
\widehat{F}_r(x)
&= \sum_{k,j}\widehat{p}_{kj}\exp({\widehat{{\bm\theta}}_r^\T}{\bf q}(x_{kj}))
\mathbf{I}(x_{kj}\le x) \\
&= n_r^{-1}\sum_{k,j} h_r(x_{kj};\widehat{{\bm\theta}})\mathbf{I}(x_{kj}\le x),
\end{align*}
where $\mathbf{I}(A)$ is the indicator of event $A$, $\widehat{{\bm\theta}}_0={\bf 0}$, and
\[
h_r(x;{\bm\theta})
= \rho_r \exp({\bm\theta}_r^\T {\bf q}(x)) / h(x;{\bm\theta}).
\]

The MELE $\widehat{{\bm\theta}}$ is asymptotically normal under the following assumptions.

\begin{assumption}
\label{as1}
\ \vspace{-2em}\\
\begin{itemize}
\item[(i)] 
The total sample size $n=\sum_{k=0}^m n_k \to \infty$, with proportions $\rho_k=n_k/n$ converging to constants for $k=0,\ldots,m$.
\item[(ii)] 
$F_0, \ldots, F_m$ satisfy the DRM \eqref{drm} with true parameter $\tilde{\bm\theta}$, 
and $\int h(x;\bm\theta)\,dF_0(x)<\infty$ for $\bm\theta$  in a neighborhood of $\tilde{\bm\theta}$. 
The components of ${\bf q}(x)$ are algebraically independent, with the first component equal to one.
\item[(iii)]
Each $F_k$ has a bounded density $f_k$ and finite second moment.
\end{itemize}
\end{assumption}

In words, the sample sizes from different populations are comparable, 
the basis functions ${\bf q}(x)$ contain no redundancy, and the relevant moments are finite.
Let $\mathbf{W}$ and $\mathbf{S}$ be $md\times md$ block matrices, 
with each block of size $d\times d$, defined by
\[
\mathbf{W}_{rs}
= \int \mathbf{q}(x)\mathbf{q}(x)^\T
\{h_r(x;\tilde{{\bm\theta}})\delta_{rs}
- h_r(x;\tilde{{\bm\theta}})h_s(x;\tilde{{\bm\theta}})\}\,d\bar{F}(x),
\quad d\bar{F}(x)=h(x;\tilde{{\bm\theta}})dF_0(x),
\]
and
\[
\mathbf{S}_{rs} = (\rho_r^{-1}\delta_{rs}+\rho_0^{-1})\,
\text{diag}\{1,0,\ldots,0\},
\quad 1\le r,s \le m,
\]
where $\delta_{rs}=1$ if $r=s$ and $0$ otherwise.

\begin{lemma}
\label{le1}
Under Assumption \ref{as1}, the MELE is asymptotically normal: as $n\to\infty$,
\[
\sqrt{n}(\widehat{{\bm\theta}}-\tilde{{\bm\theta}})
\rightsquigarrow N({\bf 0}, \mathbf{W}^{-1}-\mathbf{S}).
\]
\end{lemma}

If $\mathbf{W}^{-1}-\mathbf{S}$ were known, one could directly construct asymptotic confidence 
regions for ${\bm\theta}$ and test hypotheses such as equality of distributions. 
In practice, however, the variance depends on both $\tilde{{\bm\theta}}$ and $F_0(x)$ 
in a neat but numerically involved form, making estimation difficult.

As an alternative, one may resample from the data and compute bootstrap estimates 
$\widehat{{\bm\theta}}_b^*, b=1,2,\ldots,B$.
 The conditional distribution (given the data) of 
 $\sqrt{n}(\widehat{{\bm\theta}}^*-\widehat{{\bm\theta}})$ 
 has the same limit as that of 
$\sqrt{n}(\widehat{{\bm\theta}}-\tilde{{\bm\theta}})$. 
Thus, at modest computational cost, the bootstrap provides an accurate approximation 
without directly estimating $\mathbf{W}^{-1}-\mathbf{S}$.

Although bootstrap methods perform well in simulations, rigorous justification is essential. 
This paper takes up that task and contributes in three directions:
\begin{enumerate}
\item 
\textbf{Bootstrap consistency of parameter estimation}: 
we show that $\sqrt{n}(\widehat{{\bm\theta}}^*-\widehat{{\bm\theta}})$ given the data and 
$\sqrt{n}(\widehat{{\bm\theta}}-\tilde{{\bm\theta}})$ 
have the same limiting distribution.
\item 
\textbf{Bootstrap consistency of distribution estimation}: 
letting $\widehat{F}_r^*(x)$ denote the bootstrap MELE of $F_r(x)$, 
we show that $\sqrt{n}(\widehat{F}_r^*(x)-\widehat{F}_r(x))$ given the data 
and $\sqrt{n}(\widehat{F}_r(x)-F_r(x))$ share the same limiting process.
\item 
\textbf{Validation of existing bootstrap procedures}: 
our general results confirm the theoretical validity of many bootstrap inference 
methods for the DRM proposed in the literature.
\end{enumerate}

We present these results in the next section.

\section{Main Results}\label{se3}

We begin by introducing the function spaces and notation used throughout.
Let $\ell^\infty(\mathbb{R})$ denote the space of bounded, real-valued functions on $\mathbb{R}$,
equipped with the supremum norm $\|g\|_\infty=\sup_{x\in\mathbb{R}}|g(x)|$.
Let $L_1[0,1]$ denote the space of Lebesgue integrable functions on $[0,1]$,
equipped with the $L_1$ seminorm $\|g\|_1=\int_0^1 |g(x)|\,dx$.
The expectation and variance of $g(X)$ when $X$ has distribution $F_k$ is denoted as
$\E_k [g(X)]$ and $\var_k [g(X)]$. 
We write $\rightsquigarrow$ for weak convergence.

\subsection{Bootstrap Estimators of Parameters}
\label{se3.1}

We consider the following resampling scheme. 
For each $k=0,\dots,m$, draw a bootstrap sample
$\mathcal{X}^\ast_k=\{x^\ast_{kj}, j=1,\dots,n_k\}$ with replacement from
$\mathcal{X}_k=\{x_{kj}, j=1,\dots,n_k\}$.
Let $p^\ast_{kj}=dF_0(x^\ast_{kj})$. The bootstrap log-EL is
\[
l^\ast_n(\boldsymbol{\theta}, F_0)
=
\sum_{k,j}\log(p^\ast_{kj})+\sum_{k,j}\boldsymbol{\theta}_k^\T \mathbf{q}(x^\ast_{kj}),
\]
and the corresponding bootstrap dual log-EL is
\begin{equation}
\label{dual_boot}
l^\ast_n(\boldsymbol{\theta})
= 
-\sum_{k,j}\log\!\left\{\sum_{s=0}^m \rho_s \exp\!\big(\boldsymbol{\theta}_s^\T \mathbf{q}(x^\ast_{kj})\big)\right\}
+ 
\sum_{k,j} \boldsymbol{\theta}_k^\T \mathbf{q}(x^\ast_{kj}).
\end{equation}
The bootstrap MELE $\widehat{\boldsymbol{\theta}}^\ast$ maximizes \eqref{dual_boot}.

Next we show that $\sqrt{n}(\widehat{{\bm \theta}}^\ast-\widehat{{\bm \theta}})$ and
$\sqrt{n}(\widehat{{\bm \theta}}-\tilde{{\bm \theta}})$ 
have the same limiting distribution.
As a technical preparation, we first give a lemma with a simple proof.
Let $\mathcal{X}_n=\{\mathcal{X}_k,\ k=0,\dots,m\}$.

\begin{lemma}
\label{preparation}
 For any function \(g(\cdot)\) satisfying  \(\E_k [g^2(X)] < \infty\), for $k=0, 1, \ldots, m$.
\begin{equation}
 \label{op_star}
 \frac{1}{n}\sum_{k,j}g(x^\ast_{kj}) -\sum_{k} \rho_k\E_k [g(X)]
 =\mathbf{o}_p(1).
\end{equation}
\end{lemma}

\begin{proof}
Utilizing the finite second moment condition, it suffices to show
$$
\E
\left[\left(\frac{1}{n}\sum_{k,j}g(x^\ast_{kj})
-
\sum_{k}\rho_k\E_k\{g(X)\}\right)^2\,\bigg|\,\mathcal{X}_n\right] 
= o_p(1).
$$
It is seen
\begin{align*}
    &\E\left[\left(
    \frac{1}{n}\sum_{k,j}g(x^\ast_{kj})- \sum_{k} \rho_k\E_k [g(X)] \right)^2\,\bigg|\,\mathcal{X}_n\right] \\
    &\qquad= 
    \underbrace{\var
    \left[\frac{1}{n}\sum_{k,j}g(x^\ast_{kj})\,\bigg|\,\mathcal{X}_n\right]}_{\text{(A) Conditional Variance}} 
    + \underbrace{
    \left( \E
    \left[\frac{1}{n}\sum_{k,j}g(x^\ast_{kj})\,\bigg| \,\mathcal{X}_n\right] 
    - \sum_k \rho_k \E_k[g(X)]\right)^2}_{\text{(B) Squared Conditional Bias}}.
\end{align*}
Denote the sample mean and sample variance of $\{ g(x_{kj}), j=1, \ldots, n_k\}$ 
by  $\bar{g}_k$ and $s^2_k$.
By the law of large numbers, $\bar{g}_k \to \E_k [g(X)]$ and $s^2_k\to \var_k [g(X)]$
almost surely.
Because the bootstrap samples $x^\ast_{kj}$ are conditionally independent 
given $\mathcal{X}_n$, we have
\begin{align*}
    \var\left[\frac{1}{n}\sum_{k,j}g(x^\ast_{kj})\,\bigg|\,\mathcal{X}_n\right] 
    = \frac{1}{n^2} \sum_{k, j} \var \left[g(x^\ast_{kj})\,\big|\,\mathcal{X}_n\right] 
    = \frac{1}{n} \sum_{k=1}^m \rho_k s^2_k = o_p(1)
\end{align*}
and
$$
\E \left[\frac{1}{n}\sum_{k,j}g(x^\ast_{kj})\,\bigg|\,\mathcal{X}_n\right] 
=\sum_{k=1}^m \rho_k \bar{g}_k
 \to \sum_k \rho_k \E_k [g(X)]
$$
almost surely. 
Hence, both Terms (A) and (B) are $o_p(1)$ and the lemma is proven.
\end{proof}

\begin{theorem}
\label{the1}
Under Assumption \ref{as1},
\begin{equation}\label{theta_consis}
\sup_{\bm \theta}\left|
\PP\!\left(\sqrt{n}(\widehat{\bm \theta}^\ast-\widehat{\bm \theta}) \le {\bm \theta}\,\middle|\,{\cal X}_n\right)
- \PP({\bm Z}\le {\bm \theta})\right|
= \mathbf{o}_p(1),
\end{equation}
where ${\bm Z}\sim N({\bf 0}, \mathbf{W}^{-1}-\mathbf{S})$.
\end{theorem}

\begin{proof}
For any differentiable $f(\bm{\theta})$, write the partitioned gradient and Hessian as
\[
{\nabla f(\bm{\theta})}^\T=\Big(\tfrac{\partial f}{\partial \bm{\theta}_1}^\T,\ldots,\tfrac{\partial f}{\partial \bm{\theta}_m}^\T\Big),\qquad
\nabla^2 f(\bm{\theta})=\Big(\tfrac{\partial^2 f}{\partial \bm{\theta}_r \partial \bm{\theta}_s^\T}\Big)_{1\le r,s\le m}.
\]
For $l_n$ in \eqref{dual_like}, at $\tilde{\bm{\theta}}$,
\[
\frac{\partial l_n(\tilde{\bm{\theta}})}{\partial {\bm\theta}_r}
=\sum_{j=1}^{n_r}\!{\bf q}(x_{rj})-\sum_{k,j}\!h_r(x_{kj};\tilde{\bm{\theta}})\,{\bf q}(x_{kj})
=\sum_{k,j}\!\Big(\delta_{kr}-h_r(x_{kj};\tilde{\bm{\theta}})\Big){\bf q}(x_{kj}),
\]
\[
\frac{\partial^2 l_n(\tilde{\bm{\theta}})}{\partial \bm{\theta}_r \partial \bm{\theta}_s^\T}
=\sum_{k,j}{\bf q}(x_{kj}){\bf q}^\T(x_{kj})
\Big(\delta_{rs}h_r(x_{kj};\tilde{\bm{\theta}})-h_r(x_{kj};\tilde{\bm{\theta}})h_s(x_{kj};\tilde{\bm{\theta}})\Big).
\]
For $\|\bm{\theta}-\tilde{\bm{\theta}}\|=\mathbf{o}_p(n^{-1/3})$, 
a second-order Taylor expansion yields
\[
l_n(\bm{\theta})-l_n(\tilde{\bm{\theta}})
=\nabla l_n(\tilde{\bm{\theta}})^\T(\bm{\theta}-\tilde{\bm{\theta}})
+\tfrac12(\bm{\theta}-\tilde{\bm{\theta}})^\T \nabla^2 l_n(\tilde{\bm{\theta}}) (\bm{\theta}-\tilde{\bm{\theta}})+R_n,
\]
with $R_n=\mathbf{o}_p(1)$, since $R_n=\mathbf{O}_p\!\big(n\|\bm{\theta}-\tilde{\bm{\theta}}\|^3\big)$.
Moreover, by applying the law of large number to the Hessian blocks,
\[
\frac{1}{n}\frac{\partial^2 l_n(\tilde{\bm{\theta}})}{\partial \bm{\theta}_r \partial \bm{\theta}_s^\T}
= \mathbf{W}_{rs} + \mathbf{o}_p(1),
\qquad
\Rightarrow\qquad
\frac{1}{n}\nabla^2 l_n(\tilde{\bm{\theta}})-\mathbf{W}=\mathbf{o}_p(1).
\]
Hence the local quadratic representation
\begin{equation}
\label{ln_expansion}
l_n(\bm{\theta})-l_n(\tilde{\bm{\theta}})
= \nabla l_n(\tilde{\bm{\theta}})^\T (\bm{\theta}-\tilde{\bm{\theta}})
+ \frac{n}{2} (\bm{\theta}-\tilde{\bm{\theta}})^\T {\bf W} (\bm{\theta}-\tilde{\bm{\theta}})+\mathbf{o}_p(1).
\end{equation}
Given that \(\nabla l_n(\tilde{\bm{\theta}})\) is a sum of independent random vectors with finite moments, the central limit theorem ensures that 
 $n^{-1/2}\nabla l_n(\tilde{\bm{\theta}})\rightsquigarrow N({\bf 0},\mathbf{W}-\mathbf{WSW})$.
Hence, maximizing \eqref{ln_expansion} gives
\begin{equation}
\label{thetahat}
\sqrt{n}(\widehat{{\bm\theta}}-\tilde{{\bm \theta}})
= {\bf W}^{-1} n^{-1/2}\nabla l_n(\tilde{\bm{\theta}})+\mathbf{o}_p(1).
\end{equation}

For \(\bm{\theta} = \tilde{\bm{\theta}} + \mathbf{o}_p(n^{-1/3})\), 
the bootstrap log-likelihood \eqref{dual_boot} has Taylor expansion
\begin{equation}
l^\ast_n(\bm{\theta}) - l^\ast_n(\tilde{\bm{\theta}}) 
= 
\nabla \ell^\ast_n(\tilde{\bm{\theta}})^\T (\bm{\theta} - \tilde{\bm{\theta}}) 
+ \frac{1}{2} (\bm{\theta} - \tilde{\bm{\theta}})^\T \nabla^2 \ell^\ast_n(\tilde{\bm{\theta}}) (\bm{\theta} - \tilde{\bm{\theta}}) 
+ R^\ast_n,
\end{equation}
where \(R^\ast_n = \mathbf{o}_p(1)\).
Within this expansion, the gradient and Hessian are made of
\[
\frac{\partial l^\ast_n(\tilde{\bm{\theta}})}{\partial \bm{\theta}_r}
= \sum_{j=1}^{n_r} \mathbf{q}(x^\ast_{rj})
- \sum_{k,j} h_r(x^\ast_{kj}; \tilde{\bm{\theta}})\,\mathbf{q}(x^\ast_{kj}),
\]
\[
\frac{\partial^2 l^\ast_n(\tilde{\bm{\theta}})}{\partial \bm{\theta}_r \partial \bm{\theta}_s^\T}
= \sum_{k,j} \mathbf{q}(x^\ast_{kj}) \mathbf{q}^\T(x^\ast_{kj})
\Big( \delta_{rs} h_r(x^\ast_{kj};\tilde{\bm{\theta}})
      - h_r(x^\ast_{kj};\tilde{\bm{\theta}}) h_s(x^\ast_{kj};\tilde{\bm{\theta}}) \Big).
\]
Both of them are sum of conditional i.i.d.\ terms which validating the conclusion
of Lemma \ref{preparation}.
Hence,
\[
\frac{1}{n}\nabla^2 l^\ast_n(\tilde{\bm{\theta}})-\mathbf{W}=\mathbf{o}_p(1).
\]
A Taylor expansion then yields
\begin{equation}
\label{thetastar}
\sqrt{n}(\widehat{\bm{\theta}}^\ast-\tilde{\bm{\theta}})
= -\mathbf{W}^{-1} n^{-1/2}\nabla l^\ast_n(\tilde{\bm{\theta}}) + \mathbf{o}_p(1).
\end{equation}
Combining \eqref{thetahat}--\eqref{thetastar},
\[
\sqrt{n}(\widehat{\bm{\theta}}^\ast-\widehat{\bm{\theta}})
= \mathbf{W}^{-1} n^{-1/2}\big(\nabla l^\ast_n(\tilde{\bm{\theta}})-\nabla l_n(\tilde{\bm{\theta}})\big)+\mathbf{o}_p(1).
\]
Thus \eqref{theta_consis} follows from
\begin{equation}
\label{ln_consis}
\sup_{\bm{\theta}}
\left|
\PP\!\left\{ n^{-1/2}\big(\nabla l^\ast_n(\tilde{\bm{\theta}})-\nabla l_n(\tilde{\bm{\theta}})\big)\le \bm{\theta} \,\middle|\, \mathcal{X}_n\right\}
- \PP(\bm{Z}_1\le \bm{\theta})
\right|=\mathbf{o}_p(1),
\end{equation}
where $\bm{Z}_1\sim N(\mathbf{0},\mathbf{W}-\mathbf{WSW})$.
Since $\nabla l_n(\tilde{\bm{\theta}})$ is a sum of independent random vectors with finite moments and $\nabla l^\ast_n(\tilde{\bm{\theta}})$ is its bootstrap analogue, Theorem~23.4 of \cite{van} gives \eqref{ln_consis}, completing the proof.
\end{proof}

\subsection{Bootstrap Estimators of Distribution Functions}
\label{se3.2}

\cite{chen2013} established the asymptotic distribution of $\sqrt{n}\,\big(\widehat{F}_r(x)-F_r(x)\big)$ for fixed $x$. We extend this pointwise convergence to weak convergence of processes. 

Building on the parameter expansion in \eqref{thetahat}, the estimator $\widehat{F}_r$ admits a linearization that separates an empirical term from a smooth plug-in term involving $\widehat{\bm\theta}-\tilde{\bm\theta}$. The former yields an empirical-process component, while the latter is controlled by the asymptotic normality of $\widehat{\bm\theta}$. This decomposition leads from pointwise limits to weak convergence of the entire distribution-function process.

\begin{theorem}
\label{the2}
Under Assumption \ref{as1}, for $r=0,\ldots,m$,
\[
\sqrt{n}\,\big(\widehat{F}_r(\cdot)-F_r(\cdot)\big)\ \rightsquigarrow\ \mathcal{G}_r(\cdot)\quad\text{in }\ell^\infty(\mathbb{R}),
\]
where $\mathcal{G}_r$ is a mean-zero Gaussian process with covariance function
\begin{equation}\label{omega_r}
\omega_r(x,y)
= \sigma_r(x,y) - \rho_r^{-2}\Big\{ a_r(x\wedge y) - \mathbf{B}_r^\T(x)\,\mathbf{W}^{-1}\,\mathbf{B}_r(y) \Big\},
\end{equation}
with $x\wedge y=\min\{x,y\}$ and
\[
\sigma_r(x,y)=\rho_r^{-1}\big\{F_r(x\wedge y)-F_r(x)F_r(y)\big\},\qquad
a_r(x)=\int_{-\infty}^x\!\!\big\{h_r(t;\tilde{\bm\theta})-h_r^2(t;\tilde{\bm\theta})\big\}\,d\bar F(t),
\]
and where $\mathbf{B}_r(x)$ is an $md$-vector with $s$th $d$-segment
\[
\mathbf{B}_{r,s}(x)=\int_{-\infty}^x\!\big\{\delta_{rs}\,h_r(t;\tilde{\bm\theta})-h_r(t;\tilde{\bm\theta})h_s(t;\tilde{\bm\theta})\big\}\,\mathbf{q}(t)\,d\bar F(t),\quad s=1,\ldots,m.
\]
\end{theorem}

\begin{proof}
The proof proceeds in three steps: 
establishing asymptotic joint normality, 
working out the covariance function \(\omega_r(x, y)\), 
and verifying tightness. 
Together, these results establish the claimed weak convergence.

\medskip
\noindent\textbf{Step 1 (finite-dimensional convergence).}
From \eqref{thetahat},
\begin{equation}
\label{eq20250415-1}
\widehat{\bm\theta}-\tilde{\bm\theta} = n^{-1}\mathbf{W}^{-1}\nabla l_n(\tilde{\bm\theta}) 
+ \mathbf{o}_p(n^{-1/2}).
\end{equation}
Recall $\widehat{F}_r(x)=n_r^{-1}\sum_{k,j} h_r(x_{kj};\widehat{\bm\theta})\,\mathbf{I}(x_{kj}\le x)$.
A Taylor expansion of $h_r(\cdot;\widehat{\bm\theta})$ at $\tilde{\bm\theta}$ gives
\[
h_r(x_{kj};\widehat{\bm\theta}) = h_r(x_{kj};\tilde{\bm\theta})
+ \dot h_r(x_{kj};\tilde{\bm\theta})^\T(\widehat{\bm\theta}-\tilde{\bm\theta})
+ \mathbf{o}_p(\|\widehat{\bm\theta}-\tilde{\bm\theta}\|),
\]
with $\dot h_r=\partial h_r/\partial\bm\theta$ at $\tilde{\bm\theta}$. 
Since $\widehat{\bm\theta}-\tilde{\bm\theta}=\mathbf{O}_p(n^{-1/2})$, the remainder is $\mathbf{o}_p(n^{-1/2})$. Hence
\[
\widehat{F}_r(x)= n_r^{-1}\!\sum_{k,j}\! h_r(x_{kj};\tilde{\bm\theta})\mathbf{I}(x_{kj}\le x)
+ n_r^{-1}\!\Big\{ n^{-1}\!\sum_{k,j}\!\dot h_r(x_{kj};\tilde{\bm\theta})\mathbf{I}(x_{kj}\le x)\Big\}^\T\!\mathbf{W}^{-1}\nabla l_n(\tilde{\bm\theta})
+ \mathbf{o}_p(n^{-1/2}).
\]
By the law of large numbers,
\[
n^{-1}\!\sum_{k,j}\!\dot h_r(x_{kj};\tilde{\bm\theta})\mathbf{I}(x_{kj}\le x)
= \mathbf{B}_r(x) + \mathbf{o}_p(1).
\]
Also
\[
F_r(x)= n_r^{-1}\,\E\!\Big[\sum_{k,j} h_r(x_{kj};\tilde{\bm\theta})\mathbf{I}(x_{kj}\le x)\Big].
\]
Define
\[
\Delta_{r1}(x)=\sum_{k,j}\!\Big\{h_r(x_{kj};\tilde{\bm\theta})\mathbf{I}(x_{kj}\le x)-\E[h_r(x_{kj};\tilde{\bm\theta})\mathbf{I}(x_{kj}\le x)]\Big\},
\quad
\Delta_{r2}(x)= \mathbf{B}_r^\T(x)\,\mathbf{W}^{-1}\nabla l_n(\tilde{\bm\theta}).
\]
Then
\[
\sqrt{n}\big(\widehat{F}_r(x)-F_r(x)\big)
= \frac{\sqrt{n}}{n_r}\,\Delta_{r1}(x) + \frac{\sqrt{n}}{n_r}\,\Delta_{r2}(x) + \mathbf{o}_p(1).
\]
For any fixed $x_1,\ldots,x_t$, the vector formed by 
$\{\Delta_{r1}(x_i)\}$ and $\{\Delta_{r2}(x_i)\}$ is a sum of independent mean-zero terms; 
by the multivariate central limit theorem, 
the finite-dimensional distributions are normal with mean zero. 
We compute the limiting covariance next.

\medskip
\noindent\textbf{Step 2 (covariance function).}
Introduce
\[
c_{rk}(x)=\int_{-\infty}^x h_r(t;\tilde{\bm\theta})\,h_k(t;\tilde{\bm\theta})\,d\bar F(t),\qquad k=0,\ldots,m.
\]

(i) For $\cov\big(\Delta_{r1}(x),\Delta_{r1}(y)\big)$, direct calculation gives
\[
\cov\big(\Delta_{r1}(x),\Delta_{r1}(y)\big)
= n\Big\{ c_{rr}(x\wedge y) - \sum_{k=0}^m \rho_k^{-1} c_{rk}(x)c_{rk}(y)\Big\}.
\]

(ii) For $\cov\big(\Delta_{r1}(x),\Delta_{r2}(y)\big)$, note that
\[
\cov\big(\Delta_{r1}(x),\nabla l_n(\tilde{\bm\theta})\big)
= -\,n\sum_{k=0}^m \rho_k^{-1} c_{rk}(x)\,\mathbf{B}_k^\T.
\]
This follows from writing $\nabla l_n(\tilde{\bm\theta})$ as a sum of centered scores and using independence across samples. Multiplying by $\mathbf{W}^{-1}\mathbf{B}_r(y)$ yields
\[
\cov\big(\Delta_{r1}(x),\Delta_{r2}(y)\big)
= -\,n\sum_{k=0}^m \rho_k^{-1} c_{rk}(x)\,\mathbf{B}_k^\T \mathbf{W}^{-1}\mathbf{B}_r(y).
\]

(iii) For $\cov(\Delta_{r2}(x),\Delta_{r2}(y))$, using $\var\{\nabla l_n(\tilde{\bm\theta})\}=n(\mathbf{W}-\mathbf{WSW})$,
\[
\cov\big(\Delta_{r2}(x),\Delta_{r2}(y)\big)
= n\,\mathbf{B}_r^\T(x)\,(\mathbf{W}^{-1}-\mathbf{S})\,\mathbf{B}_r(y).
\]

Combining (i)--(iii), dividing by $n_r^2=(\rho_r n)^2$, and using the identity
$c_{rr}(x)=\rho_r F_r(x)-a_r(x)$, one obtains \eqref{omega_r}.

\medskip
\noindent\textbf{Step 3 (tightness).}
Write
\[
\tilde{F}_r(x)= n_r^{-1}\sum_{k,j} h_r(x_{kj};\tilde{\bm\theta})\,\mathbf{I}(x_{kj}\le x).
\]
For $x\ge y$,
\[
\sqrt{n}\!\left[\big(\widehat{F}_r(x)-F_r(x)\big)-\big(\widehat{F}_r(y)-F_r(y)\big)\right]
= A_n(x,y)+B_n(x,y),
\]
where
\[
A_n(x,y)=\sqrt{n}\Big[(\widehat{F}_r(x)-\tilde{F}_r(x))-(\widehat{F}_r(y)-\tilde{F}_r(y))\Big],\
B_n(x,y)=\sqrt{n}\Big[(\tilde{F}_r(x)-F_r(x))-(\tilde{F}_r(y)-F_r(y))\Big].
\]

For $A_n$, by the mean value theorem in $\bm\theta$,
\[
\sup_{|x-y|\le\delta}\!|A_n(x,y)|
\ \le\ \big(\sqrt{n}\,\|\widehat{\bm\theta}-\tilde{\bm\theta}\|\big)\,
\sup_{|x-y|\le\delta} n_r^{-1}\!\sum_{k,j}\! C\,\|\mathbf{q}(x_{kj})\|\,\mathbf{I}(y<x_{kj}\le x).
\]
Here $C$ is a generic constant. Since $\sqrt{n}\,\|\widehat{\bm\theta}-\tilde{\bm\theta}\|=\mathbf{O}_p(1)$,
it suffices to bound the empirical increment. By the Cauchy--Schwarz inequality,
\[
n_r^{-1}\!\sum_{k,j}\|\mathbf{q}(x_{kj})\|\mathbf{I}(y<x_{kj}\le x)
\ \le\ \Big(n_r^{-1}\!\sum_{k,j}\|\mathbf{q}(x_{kj})\|^2\Big)^{1/2}
\Big(n_r^{-1}\!\sum_{k,j}\mathbf{I}(y<x_{kj}\le x)\Big)^{1/2}.
\]
The first factor is $\mathbf{O}_p(1)$ by the law of large numbers, and the second factor is $\mathbf{O}_p(\delta^{1/2})$. Thus
\[
\sup_{|x-y|\le\delta}|A_n(x,y)| =\mathbf{O}_p(\delta^{1/2}).
\]

For $B_n$, it is an empirical-type process with bounded increments. Using the moment criterion (e.g., \citealp[Thm.~13.5]{billingsley2013convergence}),
\[
\E\big[|B_n(x,y)|^2\big]=\var\!\left[\sqrt{n}(\tilde{F}_r(x)-\tilde{F}_r(y))\right]
\ \le\ C\,|F_r(x)-F_r(y)| \ \le\ C'\,|x-y|,
\]
for constants $C,C'$. Hence $B_n$ is tight. 

\medskip
\noindent\textbf{Conclusion.}
We have shown that $A_n$ is tight and $B_n$ is tight, and that the finite-dimensional convergence has been established in Step~1. Together these results prove that the process
\[
\sqrt{n}\,\big(\widehat{F}_r(\cdot)-F_r(\cdot)\big)
\]
converges weakly in $\ell^\infty(\mathbb{R})$ to the Gaussian process $\mathcal{G}_r$ with covariance function \eqref{omega_r}.
\end{proof}

\begin{remark}
The covariance formula \eqref{omega_r} can be interpreted as follows. 
The leading term $\sigma_r(x,y)$ is the natural covariance of the empirical distribution function based on group-$r$ observations. 
The adjustment term 
\[
\rho_r^{-2}\Big\{a_r(x\wedge y)-\mathbf{B}_r^\T(x)\,\mathbf{W}^{-1}\,\mathbf{B}_r(y)\Big\}
\]
arises from the fact that $\widehat{F}_r$ depends on the estimated parameter $\widehat{\bm\theta}$. 
The component $a_r(\cdot)$ accounts for additional sampling variability introduced through $h_r$, while 
the quadratic form $\mathbf{B}_r^\T(x)\mathbf{W}^{-1}\mathbf{B}_r(y)$ reflects the variance reduction due to joint estimation of $\bm\theta$. 
Thus \eqref{omega_r} captures the balance between empirical fluctuations of $F_r$ and the stabilizing effect of parameter estimation.
\end{remark}

The bootstrap estimator of $F_r$ is
\begin{equation}
\label{F_boot}
\widehat{F}_r^{\ast}(x) = n_r^{-1} \sum_{k,j} h_r(x^\ast_{kj}; \widehat{\bm{\theta}}^\ast)\,\mathbf{I}(x^\ast_{kj} \le x).
\end{equation}
The next theorem shows that the bootstrap process 
$\sqrt{n}\,\big(\widehat{F}_r^{\ast}(\cdot)-\widehat{F}_r(\cdot)\big)$
shares the same weak limit as in Theorem~\ref{the2}, thereby justifying the use of resampling-based confidence regions.

\begin{theorem}
\label{the3}
Under the conditions of Theorem \ref{the1}, for $r=0,\ldots,m$, conditionally on $\mathcal{X}_n$,
\[
\sqrt{n}\,\big(\widehat{F}_r^{\ast}(\cdot)-\widehat{F}_r(\cdot)\big)\ \rightsquigarrow\ \mathcal{G}_r(\cdot)\quad\text{in }\ell^\infty(\mathbb{R}),
\]
where $\mathcal{G}_r$ is as in Theorem \ref{the2}.
\end{theorem}

\begin{proof}
The argument parallels that of Theorem~\ref{the2}, and we highlight only the main steps.

\textbf{Step 1 (finite-dimensional convergence).}
From \eqref{thetastar},
\[
\widehat{\bm{\theta}}^\ast-\tilde{\bm{\theta}}
= n^{-1}\mathbf{W}^{-1}\nabla l^\ast_n(\tilde{\bm{\theta}})+\mathbf{o}_p(n^{-1/2}).
\]
Expanding $h_r(x^\ast_{kj};\widehat{\bm{\theta}}^\ast)$ at $\tilde{\bm{\theta}}$ gives
\[
h_r(x^\ast_{kj};\widehat{\bm{\theta}}^\ast)
= h_r(x^\ast_{kj};\tilde{\bm{\theta}})
+ \dot h_r(x^\ast_{kj};\tilde{\bm{\theta}})^\T(\widehat{\bm{\theta}}^\ast-\tilde{\bm{\theta}})
+ \mathbf{o}_p(n^{-1/2}),
\]
so that
\[
\widehat{F}^\ast_r(x)= n_r^{-1}\!\sum_{k,j}\! h_r(x^\ast_{kj};\tilde{\bm{\theta}})\mathbf{I}(x^\ast_{kj}\le x)
+ n_r^{-1}\!\Big\{n^{-1}\!\sum_{k,j}\! \dot h_r(x^\ast_{kj};\tilde{\bm{\theta}})\mathbf{I}(x^\ast_{kj}\le x)\Big\}^\T\!\mathbf{W}^{-1}\nabla l^\ast_n(\tilde{\bm{\theta}})
+ \mathbf{o}_p(n^{-1/2}).
\]
By \eqref{op_star},
\[
n^{-1}\!\sum_{k,j}\! \dot h_r(x^\ast_{kj};\tilde{\bm{\theta}})\mathbf{I}(x^\ast_{kj}\le x)
= \mathbf{B}_r(x)+\mathbf{o}_p(1).
\]
Also $\widehat{F}_r(x)=n_r^{-1}\sum_{k,j} h_r(x_{kj};\widehat{\bm{\theta}})\mathbf{I}(x_{kj}\le x)$.
Hence
\[
\widehat{F}^\ast_r(x)-\widehat{F}_r(x)= n_r^{-1}\Delta^\ast_{r1}(x) + n_r^{-1}\Delta^\ast_{r2}(x) + \mathbf{o}_p(n^{-1/2}),
\]
where
\[
\Delta^\ast_{r1}(x)=\sum_{k,j}\!\Big\{h_r(x^\ast_{kj};\tilde{\bm{\theta}})\mathbf{I}(x^\ast_{kj}\le x)
- h_r(x_{kj};\tilde{\bm{\theta}})\mathbf{I}(x_{kj}\le x)\Big\},
\;
\Delta^\ast_{r2}(x)=\mathbf{B}_r^\T(x)\,\mathbf{W}^{-1}\big\{\nabla l^\ast_n(\tilde{\bm{\theta}})-\nabla l_n(\tilde{\bm{\theta}})\big\}.
\]
By Theorem~23.4 of \cite{van}, conditionally on $\mathcal{X}_n$ the finite-dimensional distributions of $n^{-1/2}\Delta^\ast_{ri}(\cdot)$ ($i=1,2$) converge to those of their population counterparts, yielding the same limiting covariance $\omega_r(x,y)$ as in Theorem~\ref{the2}.

\medskip
\textbf{Step 2 (tightness).}
Let $\tilde{F}^\ast_r(x)=n_r^{-1}\sum_{k,j} h_r(x^\ast_{kj};\tilde{\bm{\theta}})\mathbf{I}(x^\ast_{kj}\le x)$.
For $x\ge y$,
\[
\sqrt{n}\!\left[\big(\widehat{F}^\ast_r(x)-\widehat{F}_r(x)\big)-\big(\widehat{F}^\ast_r(y)-\widehat{F}_r(y)\big)\right]
= A^\ast_n(x,y)+B^\ast_n(x,y),
\]
with
\[
A^\ast_n(x,y)=\sqrt{n}\big[(\widehat{F}^\ast_r-\tilde{F}^\ast_r)(x)-(\widehat{F}^\ast_r-\tilde{F}^\ast_r)(y)\big],\quad
B^\ast_n(x,y)=\sqrt{n}\big[(\tilde{F}^\ast_r-F_r)(x)-(\tilde{F}^\ast_r-F_r)(y)\big].
\]

For $A^\ast_n$, conditioning on $\mathcal{X}_n$ and using the mean value theorem,
\[
\sup_{|x-y|\le\delta}|A^\ast_n(x,y)| \le \big(\sqrt{n}\,\|\widehat{\bm\theta}^\ast-\tilde{\bm\theta}\|\big)\,
\sup_{|x-y|\le\delta} n_r^{-1}\sum_{k,j} C\,\|\mathbf{q}(x^\ast_{kj})\|\mathbf{I}(y<x^\ast_{kj}\le x).
\]
By \eqref{thetastar}, $\sqrt{n}\,\|\widehat{\bm\theta}^\ast-\tilde{\bm\theta}\|=\mathbf{O}_p(1)$.
For the empirical increment, apply Cauchy--Schwarz:
\[
n_r^{-1}\sum_{k,j}\|\mathbf{q}(x^\ast_{kj})\|\mathbf{I}(y<x^\ast_{kj}\le x)
\ \le\ \Big(n_r^{-1}\sum_{k,j}\|\mathbf{q}(x^\ast_{kj})\|^2\Big)^{1/2}
\Big(n_r^{-1}\sum_{k,j}\mathbf{I}(y<x^\ast_{kj}\le x)\Big)^{1/2}.
\]
The first factor is $O_p(1)$ conditionally, and the second is $O_p(\delta^{1/2})$. Thus
\[
\sup_{|x-y|\le\delta}|A^\ast_n(x,y)|=O_p(\delta^{1/2}).
\]

For $B^\ast_n$, a conditional variance bound analogous to Step~3 of Theorem~\ref{the2} gives
\[
\E^\ast\big[|B^\ast_n(x,y)|^2\big]\ \le\ C\,|\widehat{F}_r(x)-\widehat{F}_r(y)| \ \le\ C'\,|x-y|,
\]
in probability, where $\E^\ast$ denotes conditional expectation given $\mathcal{X}_n$.
Thus $B^\ast_n$ is conditionally tight. 

\medskip
\textbf{Conclusion.}
Both $A^\ast_n$ and $B^\ast_n$ are tight, and Step~1 established conditional finite-dimensional convergence with the same covariance as in Theorem~\ref{the2}. Therefore,
\[
\sqrt{n}\,\big(\widehat{F}_r^\ast(\cdot)-\widehat{F}_r(\cdot)\big)
\]
converges weakly in $\ell^\infty(\mathbb{R})$ to $\mathcal{G}_r$, validating the bootstrap procedure.
\end{proof}

\begin{remark}
Theorem~\ref{the3} establishes that the bootstrap reproduces the same Gaussian limit as the original process. 
This guarantees that confidence bands or regions for $F_r$ constructed by resampling are asymptotically valid, 
providing a practical inference tool that requires no further analytic approximation beyond Theorem~\ref{the2}.
\end{remark}

\section{Applications}
\label{se4}
We apply the preceding theory to several inference problems under the density ratio model (DRM), showing that the bootstrap procedures used in prior work are theoretically valid and consistent with our results.

\subsection{Quantile Functions}
We define the quantile function as $Q(p)=\inf\{x:F(x)\ge p\}$, and denote by $\widehat{Q}_r(p)$ and $\widehat{Q}_r^\ast(p)$ the DRM and bootstrap estimators of $Q_r(p)$, respectively. While \cite{chen2013} obtained the limiting distribution of $\widehat{Q}_r(p)$ (via a Bahadur representation) at fixed $p$, our framework extends this to the entire quantile process. Building on Theorem~\ref{the2}, we establish weak convergence of $\sqrt{n}(\widehat{Q}_r-Q_r)$ and its bootstrap analog in $L_1[0,1]$.

\medskip
\noindent
\emph{Context.} Working in $L_1[0,1]$ is both natural and useful: for unbounded supports, $Q_r$ itself can be unbounded, yet most target statistics are integrals of $Q_r$ over $p\in(0,1)$. Weak convergence in $L_1[0,1]$ is therefore sufficient (via the continuous mapping theorem) for inference on such functionals. 

A direct application of the functional delta method from $\ell^\infty(\mathbb{R})$ to $L_1[0,1]$ is obstructed because the inverse map $\phi:F\mapsto Q$ is not Hadamard differentiable under the sup norm \citep{kaji}. To circumvent this, we proceed in two steps. First, we strengthen Theorem~\ref{the2} to obtain weak convergence of $\sqrt{n}(\widehat{F}_r-F_r)$ in the space
\[
\mathbb{L}=\Big\{f:\mathbb{R}\to\mathbb{R}\ \text{bounded and integrable}\Big\},\qquad
\|f\|_{\mathbb{L}}:=\|f\|_\infty+\int_{\mathbb{R}}|f(x)|\,dx.
\]
Second, we invoke \cite{kaji}, Theorem 1.3, which shows that $\phi$ is Hadamard differentiable from $\mathbb{L}$ to $L_1[0,1]$, thereby enabling the functional delta method. The bootstrap result follows analogously.

\begin{theorem}
\label{the4}
Under the conditions of Theorem \ref{the1}, for $r=0,\ldots,m$ and $p\in(0,1)$,
\begin{equation}
\label{quant}
\sqrt{n}\,\big(\widehat{Q}_r(p)-Q_r(p)\big)\ \rightsquigarrow\ -\,Q_r'(p)\,\mathcal{G}_r\!\big(Q_r(p)\big)\quad\text{in }L_1[0,1],
\end{equation}
where $Q_r'(p)$ is the derivative of $Q_r(p)$, and $\mathcal{G}_r$ is from Theorem~\ref{the2}. Furthermore, conditionally on $\mathcal{X}_n$,
\[
\sqrt{n}\,\big(\widehat{Q}_r^\ast(p)-\widehat{Q}_r(p)\big)\ \rightsquigarrow\ -\,Q_r'(p)\,\mathcal{G}_r\!\big(Q_r(p)\big)\quad\text{in }L_1[0,1].
\]
\end{theorem}

\begin{proof}
\textbf{(1) Convergence in $\mathbb{L}$.}\ 
By Theorem~\ref{the2}, $\sqrt{n}(\widehat{F}_r-F_r)\rightsquigarrow\mathcal{G}_r$ in $\ell^\infty(\mathbb{R})$. To strengthen this to $\mathbb{L}$, we verify the equal-integrability condition of \cite{kaji}, Theorem 1.1: for every $\eta,\epsilon>0$ there exists $\delta>0$ such that
\begin{equation}
\label{eq4.1.1}
\limsup_{n\to\infty}\PP\!\left(\sup_x \int_{x-\delta}^{x+\delta}\Big|\sqrt{n}\big(\widehat{F}_r(y)-F_r(y)\big)\Big|\,dy>\eta\right)\le\epsilon.
\end{equation}
We illustrate for $r=0$; other groups are analogous. Write
\[
\widehat{F}_0-F_0=(\widehat{F}_0-\tilde{F}_0)+(\tilde{F}_0-F_0),
\qquad
\tilde{F}_0(x)=n_0^{-1}\sum_{k,j}h_0(x_{kj};\tilde{\bm\theta})\,\mathbf{I}(x_{kj}\le x).
\]
A mean-value expansion of $h_0(\cdot;\widehat{\bm\theta})$ about $\tilde{\bm\theta}$ shows
\[
\sup_y|\widehat{F}_0(y)-\tilde{F}_0(y)| 
= \mathbf{O}_p(\|\widehat{\bm\theta}-\tilde{\bm\theta}\|) 
= \mathbf{O}_p(n^{-1/2}),
\]
under $\E\|\mathbf{q}(X)\|<\infty$. Standard empirical-process bounds yield $\sup_y|\tilde{F}_0(y)-F_0(y)|=O_p(n^{-1/2})$. Consequently,
\[
\sup_x\int_{x-\delta}^{x+\delta}\!\Big|\sqrt{n}\big(\widehat{F}_0(y)-F_0(y)\big)\Big|\,dy
\ \le\ 2\sqrt{n}\delta\,\sup_y|\widehat{F}_0(y)-F_0(y)|=\mathbf{O}_p(\delta).
\]
Choosing $\delta$ sufficiently small ensures \eqref{eq4.1.1}. Thus $\sqrt{n}(\widehat{F}_r-F_r)\rightsquigarrow\mathcal{G}_r$ in $\mathbb{L}$.

\medskip
\textbf{(2) Functional delta method.}\ 
By \cite{kaji}, Theorem 1.3, the map $\phi:\mathbb{L}\to L_1[0,1]$, $F\mapsto Q$, is Hadamard differentiable at $F_r$ with derivative
\[
\phi'_{F_r}(h)(p)= -\frac{h(Q_r(p))}{F_r'(Q_r(p))} = -\,Q_r'(p)\,h(Q_r(p)).
\]
Applying the functional delta method \citep[Theorem 2.8]{Kosorok2008} to $\sqrt{n}(\widehat{F}_r-F_r)$ yields \eqref{quant}.

\medskip
\textbf{(3) Bootstrap.}\ 
A conditional analog of Step (1) (using bootstrap maximal inequalities; e.g., \citealt{ahmed2001almost}) shows $\sqrt{n}(\widehat{F}_r^\ast-\widehat{F}_r)\rightsquigarrow \mathcal{G}_r$ in $\mathbb{L}$, conditionally on $\mathcal{X}_n$. The bootstrap functional delta method \citep[Theorem 2.9]{Kosorok2008} then yields the stated $L_1[0,1]$ limit for $\sqrt{n}(\widehat{Q}_r^\ast-\widehat{Q}_r)$.
\end{proof}

\begin{remark}
Theorem~\ref{the4} extends classical quantile asymptotics to the DRM setting at the process level. 
The appearance of $-Q_r'(p)\,\mathcal{G}_r(Q_r(p))$ in \eqref{quant} mirrors the familiar Bahadur representation, 
but now the result holds uniformly in $p$ and in the $L_1$ sense. 
The bootstrap statement ensures that resampling-based inference for quantiles inherits the same validity as for distribution functions, 
making procedures such as percentile bootstrap confidence bands theoretically justified.
\end{remark}

\subsection{Dominance Index}
Let $F_0$ and $F_1$ be two DRM distributions with densities $f_0,f_1$ and quantile functions $Q_0,Q_1$. The dominance index of $F_0$ over $F_1$ is
\[
\tilde{\gamma}:=\gamma(F_0,F_1)=\mu\{\,p\in(0,1): Q_0(p)>Q_1(p)\,\},
\]
where $\mu$ is Lebesgue measure. As it aggregates where $F_0$ exceeds $F_1$ across quantile levels, $\tilde{\gamma}$ is robust to extremes and interpretable through stochastic-dominance intuition. When $F_0$ and $F_1$ cross finitely many times, \citet{zhuang} consider the plug-in estimator
\[
\widehat{\gamma}=\gamma(\widehat{F}_0,\widehat{F}_1)=\mu\{\,p\in(0,1): \widehat{Q}_0(p)>\widehat{Q}_1(p)\,\}.
\]
We aim to match the limiting distribution of $\widehat{\gamma}$ with its bootstrap counterpart.

\begin{lemma}
\label{le2}
(\citealp{zhuang})
Assume $F_0(x)=F_1(x)$ at finitely many $x$, with $f_0(x)\neq f_1(x)$ at each crossing. 
Let $[x_L,x_U]$ be the common support of $F_0$ and $F_1$, and suppose
\begin{equation}
\label{as2}
\lim_{x\to x_L}\frac{f_1(x)}{f_0(x)}\neq 1,\qquad \lim_{x\to x_U}\frac{f_1(x)}{f_0(x)}\neq 1.
\end{equation}
Then $n^{1/2}\{\widehat{\gamma}-\tilde{\gamma}\}\rightsquigarrow N(0,\sigma^2)$ for a variance $\sigma^2$ (available in closed form, though algebraically involved).
\end{lemma}

\medskip
\noindent\emph{Remark (motivation).}
While Lemma~\ref{le2} yields a CLT, $\sigma^2$ depends on unknown features of $(F_0,F_1)$ and is cumbersome to estimate directly. Leveraging our process-level results for $F_r$ and $Q_r$ (Theorems~\ref{the2}--\ref{the4}), we instead validate a bootstrap that \emph{automatically} reproduces the limit law of $\widehat{\gamma}$, avoiding explicit variance estimation.

Next, we give two additional lemmas.

\begin{lemma}[Localization of bootstrap crossings]
\label{le3}
Assume the conditions of Lemma~\ref{le2}. Work under the probability integral transform so that $F_0(u)=u$ on $[0,1]$, and \eqref{as2} becomes $f_1(0)\neq 1$ and $f_1(1)\neq 1$. Let $b_n:=n^{-1/2}\log n \to 0$. If $\tilde{x}$ is the first crossing point of $F_0$ and $F_1$ in $(0,1)$, then, with probability tending to one (conditionally for the bootstrap),
\[
\widehat{F}^\ast_0 \ \text{and}\ \widehat{F}^\ast_1 \ \text{do not cross on}\ (b_n,\ \tilde{x}-\epsilon)
\]
for any fixed small $\epsilon>0$. Symmetrically, if $\tilde{y}$ is the last crossing in $(0,1)$, then there is no bootstrap crossing on $(\tilde{y}+\epsilon,\ 1-b_n)$ with probability tending to one.
\end{lemma}

\begin{proof}
We prove the claim to the left of the first crossing; the right-side statement is analogous. Since $F_0(u)=u$ and $f_1(0)\neq 1$, continuity implies the existence of $\delta>0$ and $\eta\in(0,\epsilon)$ such that
\[
F_1(u)-F_0(u)=F_1(u)-u \ \ge\ \delta\,u \quad \text{for all } u\in[0,\eta].
\]
From Theorem~\ref{the2} and its bootstrap version (Theorem~\ref{the3}), we have the uniform conditional rate
\[
\sup_{u\in[0,1]}\big|\widehat{F}_r^\ast(u)-F_r(u)\big|=\mathbf{O}_p(n^{-1/2}),\qquad r=0,1,
\]
hence, uniformly in $u\in[0,1]$,
\[
\widehat{F}^\ast_1(u)-\widehat{F}^\ast_0(u)
=\{F_1(u)-F_0(u)\}+R_n^\ast(u),
\qquad \sup_{u}|R_n^\ast(u)|=\mathbf{O}_p(n^{-1/2}).
\]
Fix $n$ large so that $b_n<\eta$ and $|R_n^\ast(u)|\le \tfrac{1}{2}\delta b_n$ with conditional probability at least $1-\varepsilon$ (for any fixed $\varepsilon>0$). Then, for all $u\in[b_n,\eta]$,
\[
\widehat{F}^\ast_1(u)-\widehat{F}^\ast_0(u)
\ \ge\ \delta u - \tfrac{1}{2}\delta b_n \ \ge\ \tfrac{1}{2}\delta b_n \ >\ 0.
\]
On $[\eta,\tilde{x}-\epsilon]$, $F_1-F_0$ is bounded away from $0$ (by definition of the first crossing), so the same uniform $O_p(n^{-1/2})$ control implies $\widehat{F}^\ast_1-\widehat{F}^\ast_0>0$ there for $n$ large. Hence no bootstrap crossing occurs on $(b_n,\tilde{x}-\epsilon)$ with probability tending to one.
\end{proof}

\begin{lemma}[Uniqueness of bootstrap crossing near each true crossing]
\label{le4}
Assume Lemma~\ref{le3}. 
Let $\tilde{x}\in(0,1)$ be a crossing of $(F_0,F_1)$ such that $f_1(\tilde{x})\neq f_0(\tilde{x})$. Then there exists $\eta>0$ such that, with probability tending to one (conditionally for the bootstrap), the function $u\mapsto \widehat{F}^\ast_1(u)-\widehat{F}^\ast_0(u)$ is strictly monotone on $(\tilde{x}-\eta,\tilde{x}+\eta)$ and therefore has a unique zero in this interval. The same conclusion holds for $u\mapsto \widehat{F}_1(u)-\widehat{F}_0(u)$.
\end{lemma}

\begin{proof}
Let $G(u):=F_1(u)-F_0(u)$. By assumption, $G(\tilde{x})=0$ and $G'(\tilde{x})=f_1(\tilde{x})-f_0(\tilde{x})\neq 0$. By continuity of $G'$, there exist $\eta>0$ and $c>0$ such that either $G'(u)\ge c>0$ for all $u\in(\tilde{x}-\eta,\tilde{x}+\eta)$ or $G'(u)\le -c<0$ there. Hence $G$ is strictly monotone on $(\tilde{x}-\eta,\tilde{x}+\eta)$ and has a unique zero at $u=\tilde{x}$.

Write the bootstrap perturbation as
\[
\widehat{G}^\ast(u):=\widehat{F}^\ast_1(u)-\widehat{F}^\ast_0(u)
= G(u)+R_n^\ast(u),
\qquad \sup_{u\in[0,1]}|R_n^\ast(u)|=\mathbf{O}_p(n^{-1/2}).
\]
Shrink $\eta$ if necessary so that $|G(u)|\ge c_1|u-\tilde{x}|$ on $(\tilde{x}-\eta,\tilde{x}+\eta)$ for some $c_1>0$ (mean value theorem with $G'$ bounded away from $0$). For $n$ large, with probability tending to one we have $\sup_{|u-\tilde{x}|\le \eta}|R_n^\ast(u)|\le \tfrac{1}{2}c_1\eta$. Then, for $u_1<u_2$ in $(\tilde{x}-\eta,\tilde{x}+\eta)$,
\[
\widehat{G}^\ast(u_2)-\widehat{G}^\ast(u_1)
=\{G(u_2)-G(u_1)\}+\{R_n^\ast(u_2)-R_n^\ast(u_1)\}.
\]
Since $G$ is strictly monotone, $G(u_2)-G(u_1)$ has fixed sign and magnitude at least $c_1(u_2-u_1)$, which dominates the $\mathbf{o}_p(1)$ perturbation $R_n^\ast(u_2)-R_n^\ast(u_1)$. Thus $\widehat{G}^\ast$ is strictly monotone on $(\tilde{x}-\eta,\tilde{x}+\eta)$ for $n$ large, and it has exactly one zero there. The same argument (without conditional probability) applies to $\widehat{G}(u):=\widehat{F}_1(u)-\widehat{F}_0(u)$.
\end{proof}

\medskip
Define the bootstrap estimator
\[
\widehat{\gamma}^{\ast}=\gamma(\widehat{F}_0^{\ast},\widehat{F}_1^{\ast})
=\mu\{\,p\in(0,1):\widehat{Q}_0^{\ast}(p)>\widehat{Q}_1^{\ast}(p)\,\},
\]
where $\widehat{F}_r^\ast,\widehat{Q}_r^\ast$ are the DRM bootstrap counterparts.

\begin{theorem}
\label{the6}
Under the conditions of Lemma \ref{le2},
\begin{equation}\label{gamma_consis}
\sup_{x}\Big| \PP\!\big(\sqrt{n}(\widehat{\gamma}^{\ast}-\widehat{\gamma})\le x \,\big|\, {\cal X}_n\big)- \PP(Z_{\gamma}\le x)\Big|=\mathbf{o}_p(1),
\end{equation}
where $Z_{\gamma}\sim N(0, \sigma^2)$ and $\sigma^2$ is as in Lemma \ref{le2}.
\end{theorem}

\begin{proof}
\emph{Roadmap.} We (i) reduce by the PIT normalization used above, (ii) localize bootstrap crossings to the boundary strips or small neighborhoods of the true crossings (Lemmas~\ref{le3}--\ref{le4}), (iii) linearize crossing locations via an argmax expansion driven by the quantile processes (Theorem~\ref{the4}), and (iv) conclude by conditional weak convergence of the bootstrap process.

\medskip
\noindent\emph{Reduction via PIT.}
As above, work on $[0,1]$ with $F_0(u)=u$ and $f_1(0)\neq 1\neq f_1(1)$.

\medskip
\noindent\emph{Localization of bootstrap crossings.}
Let $b_n:=n^{-1/2}\log n$ as in Lemma~\ref{le3}. The lemmas confine any bootstrap crossings of $\widehat F_0^\ast$ and $\widehat F_1^\ast$ to $(0,b_n)\cup(1-b_n,1)$ or to small neighborhoods of the true crossings, and ensure uniqueness there.

\medskip
\noindent\emph{Case 1: no crossings.}
If $F_0$ and $F_1$ never cross, then $\tilde{\gamma}\in\{0,1\}$. By Lemma~\ref{le3}, any bootstrap crossing can only occur in boundary strips of length $O(b_n)$, so $\widehat{\gamma}^\ast=\mathbf{o}_p(n^{-1/2})$ or $1-\mathbf{o}_p(n^{-1/2})$. The same holds for $\widehat{\gamma}$ (cf. the supplement of \citealp{zhuang}). Thus
\[
n^{1/2}(\widehat{\gamma}^\ast-\widehat{\gamma})=\mathbf{o}_p(1),
\]
and \eqref{gamma_consis} is immediate.

\medskip
\noindent\emph{Case 2: one crossing.}
Suppose $F_0$ and $F_1$ cross once at $\tilde{x}_1\in(0,1)$, so $Q_0(t)>Q_1(t)$ iff $t<\tilde{x}_1$. By Lemma~\ref{le4} (and the analogous result for $\widehat{F}_r$), there exist unique $\widehat{x}_1,\widehat{x}_1^\ast\in(b_n,1-b_n)$ with
\[
\widehat{\gamma}=\widehat{x}_1+\mathbf{o}_p(n^{-1/2}),\qquad
\widehat{\gamma}^\ast=\widehat{x}_1^\ast+\mathbf{o}_p(n^{-1/2}).
\]
For small $\epsilon>0$, define the integrated contrasts
\[
\Phi(u)=\int_{\tilde{x}_1-\epsilon}^{u}\{Q_0(t)-Q_1(t)\}\,dt,\quad
\widehat{\Phi}_n(u)=\int_{\tilde{x}_1-\epsilon}^{u}\{\widehat{Q}_0(t)-\widehat{Q}_1(t)\}\,dt,
\]
\[
\widehat{\Phi}_n^\ast(u)=\int_{\tilde{x}_1-\epsilon}^{u}\{\widehat{Q}_0^\ast(t)-\widehat{Q}_1^\ast(t)\}\,dt.
\]
By Theorem~\ref{the4}, the centered processes
\[
W_n(t)=n^{1/2}\!\Big[(\widehat Q_0-\widehat Q_1)(t)-(Q_0-Q_1)(t)\Big],\quad
W_n^\ast(t)=n^{1/2}\!\Big[(\widehat Q_0^\ast-\widehat Q_1^\ast)(t)-(\widehat Q_0-\widehat Q_1)(t)\Big]
\]
converge weakly (conditionally for $W_n^\ast$) in $L_1[0,1]$ to the same Gaussian limit. Hence
\[
\Delta_n(u):=\widehat{\Phi}_n(u)-\Phi(u)=n^{-1/2}\!\int_{\tilde{x}_1-\epsilon}^{u}W_n(t)\,dt,\quad
\Delta_n^\ast(u):=\widehat{\Phi}_n^\ast(u)-\widehat{\Phi}_n(u)=n^{-1/2}\!\int_{\tilde{x}_1-\epsilon}^{u}W_n^\ast(t)\,dt.
\]
Because $\Phi$ has a unique maximizer at $\tilde{x}_1$ with negative second derivative there
($f_0(Q_0(\tilde{x}_1))\neq f_1(Q_1(\tilde{x}_1))$), the argmax map is Hadamard differentiable. By Theorem A.2 of \citet{alvarez2017models},
\[
n^{1/2}(\widehat{x}_1-\tilde{x}_1)=-C\,W_n(\tilde{x}_1)+\mathbf{o}_p(1),\qquad
n^{1/2}(\widehat{x}_1^\ast-\widehat{x}_1)=-C\,\{W_n^\ast(\tilde{x}_1)-W_n(\tilde{x}_1)\}+\mathbf{o}_p(1),
\]
with
\[
C=\frac{f_0(Q_0(\tilde{x}_1))\,f_1(Q_1(\tilde{x}_1))}{\,f_0(Q_0(\tilde{x}_1))-f_1(Q_1(\tilde{x}_1))\,}.
\]
Therefore,
\[
n^{1/2}(\widehat{\gamma}^\ast-\widehat{\gamma})
=-C\big\{W_n^\ast(\tilde{x}_1)-W_n(\tilde{x}_1)\big\}+\mathbf{o}_p(1).
\]
By Theorem~\ref{the4} (bootstrap convergence of the quantile process) together with 
Theorem~\ref{the1} (bootstrap consistency for the parameter estimation embedded in 
$\widehat Q_r^\ast$), the conditional limit of $W_n^\ast(\tilde{x}_1)-W_n(\tilde{x}_1)$ 
equals the (unconditional) limit of $W_n(\tilde{x}_1)$. 
Hence $n^{1/2}(\widehat{\gamma}^\ast-\widehat{\gamma})\rightsquigarrow N(0,\sigma^2)$ 
conditionally on ${\cal X}_n$, where $\sigma^2=C^2\var\{W(\tilde{x}_1)\}$ matches Lemma~\ref{le2}.

\medskip
\noindent\emph{Case 3: finitely many crossings.}
Let the crossings be $\tilde{x}_1,\ldots,\tilde{x}_K$; then
\[
\tilde{\gamma}=\sum_{j=1}^K s_j \tilde{x}_j,\qquad s_j\in\{+1,-1\},
\]
with $s_j$ determined by the local sign of $Q_0-Q_1$. 
By Lemmas~\ref{le3}--\ref{le4}, there are unique 
$\widehat{x}_j,\widehat{x}_j^\ast$ near each $\tilde{x}_j$, and
\[
\widehat{\gamma}=\sum_{j=1}^K s_j \widehat{x}_j+\mathbf{o}_p(n^{-1/2}),\qquad
\widehat{\gamma}^\ast=\sum_{j=1}^K s_j \widehat{x}_j^\ast+\mathbf{o}_p(n^{-1/2}).
\]
Applying the same argmax linearization uniformly over $j$ yields
\[
n^{1/2}(\widehat{x}_j^\ast-\widehat{x}_j)=-C_j\{W_n^\ast(\tilde{x}_j)-W_n(\tilde{x}_j)\}+\mathbf{o}_p(1),
\]
with $C_j$ defined as above at $\tilde{x}_j$. Summing over $j$,
\[
n^{1/2}(\widehat{\gamma}^\ast-\widehat{\gamma})
=-\sum_{j=1}^K s_j C_j\{W_n^\ast(\tilde{x}_j)-W_n(\tilde{x}_j)\}+\mathbf{o}_p(1).
\]
By joint (bootstrap) weak convergence of the quantile process at finitely many points (Theorem~\ref{the4}), the conditional limit coincides with the Gaussian limit of $\sum_{j=1}^K s_j C_j W(\tilde{x}_j)$, giving \eqref{gamma_consis}.
\end{proof}

\begin{remark}
Theorem~\ref{the6} confirms that the bootstrap reproduces the sampling law of $\widehat{\gamma}$ without explicit estimation of the intricate variance in Lemma~\ref{le2}. In practice, percentile or basic bootstrap intervals for $\gamma(F_0,F_1)$ are therefore asymptotically valid, and the approach aligns seamlessly with the process-level results established in Theorems~\ref{the2}--\ref{the4}.
\end{remark}

\section{Simulation}
\label{se5}
We assess the finite-sample performance of the bootstrap procedures developed in Theorems~\ref{the2}--\ref{the4}. Specifically, we examine the empirical coverage of percentile bootstrap confidence intervals (CIs) for (i) the DRM parameter vector $\bm{\theta}$, (ii) the distribution value $F_r(x)$ at selected $x$, and (iii) the quantile $Q_r(p) = F_r^{-1}(p)$ at selected $p$, targeting the nominal level $1-\alpha=0.95$.

\medskip
\noindent\textbf{Generic target and bootstrap CI.}\ Let $\xi$ be any \emph{scalar functional} of $(F_0,\ldots,F_m)$--for example, a component $\theta_{rs}$ of $\bm{\theta}$, a value $F_r(x)$ at a fixed $x$, or a quantile $F_r^{-1}(p)$ at a fixed $p$. Denote its true value by $\widetilde{\xi}$, the DRM estimator by $\widehat{\xi}$, and the $b$th bootstrap replicate by $\widehat{\xi}^{\ast b}$. By the bootstrap consistency established earlier, the laws of $\widehat{\xi}-\widetilde{\xi}$ and $\widehat{\xi}^{\ast}-\widehat{\xi}$ share the same limit. If $q_\alpha$ denotes the empirical $\alpha$--quantile of $\{\widehat{\xi}^{\ast b}-\widehat{\xi}: b=1,\ldots,B\}$, then
\[
\PP\!\big(\widehat{\xi}-\widetilde{\xi}\le q_\alpha\big)\approx \alpha.
\]
This yields the standard \emph{percentile} CI for $\xi$:
\begin{equation}
\label{boot.CI}
\mathrm{CI}_{1-\alpha}(\xi)
=\Big[\ \widehat{\xi}-q_{1-\alpha/2}\,,\ \widehat{\xi}-q_{\alpha/2}\ \Big].
\end{equation}

\medskip
\noindent\textbf{Design, resampling and evaluation.}\ In each scenario we generate i.i.d.\ observations $\{x_{rj}\}_{j=1}^{n_r}\sim F_r$ for groups $r=0,\ldots,m$, collected as $\mathcal{X}_n=\{x_{rj}\}$. For each Monte Carlo run, we compute $\widehat{\xi}$ for the following targets
\begin{itemize}\setlength{\itemsep}{1pt}
\item DRM parameters $\theta_{rs}$ with $r=1,\ldots,m$ and $s=1,\ldots,d$;
\item distribution values $F_r(x)$ at $x= F_0^{-1}(p):  p=0.1, \ldots, 0.9 $;
\item quantiles $Q_r(p) = F_r^{-1}(p)$ at $p = 0.1, \ldots, 0.9$.
\end{itemize}
To approximate sampling variability, we perform $B=999$ bootstrap resamples by drawing with replacement within each group $r$ (preserving $\{n_r\}$), recompute $\widehat{\xi}^{\ast b}$, and form the endpoints in \eqref{boot.CI}. Empirical coverage rates are based on $N=2000$ Monte Carlo runs.

\medskip
\noindent\textbf{Scenarios and DRM bases.}\ We consider two data--generating settings chosen to probe distinct shapes and tail behaviors.
\begin{enumerate}\setlength{\itemsep}{1pt}
\item 
\emph{Gamma family ($m=4$):} 
$F_r\in\{\Gamma(5,1.5),\Gamma(5,1.4),\Gamma(6,1.3),\Gamma(6,1.2),\Gamma(7,1.1)\}$ with sample sizes $(500,450,550,650,675)$; 
DRM basis $\mathbf{q}(x)=(1,\,x,\,\log x)^\T$ ($d=3$).
\item 
\emph{Normal family ($m=6$):} 
$F_r \in\{ N(11,1),N(11.5,2),N(12,3),N(12.5,4),N(13,5)$,\\
$N(13.5,6),N(14,7)\}$ with sample sizes $(300,320,340,330,350,370,400)$; 
DRM basis $\mathbf{q}(x)=(1,\,x,\,x^2)^\T$ ($d=3$).
\end{enumerate}
Here $N(\mu,\sigma^2)$ has mean $\mu$ and variance $\sigma^2$, and $\Gamma(\alpha,\beta)$ has shape $\alpha$ and scale $\beta$.

\medskip
\noindent\textbf{Findings.}\ The simulation results in given in Tables \ref{tab1}, \ref{tab2} and \ref{tab3}. Across both scenarios, the bootstrap CIs for the \emph{model parameters} $\theta_{rs}$ and the \emph{distribution values} $F_r(x)$ attain empirical coverage close to $0.95$, typically within $\pm\,0.02$. For \emph{quantiles} $Q_r(p)$, coverage is also near nominal with mild undercoverage at tail levels (e.g., $p\in\{0.1,0.9\}$ in the Normal setting), a familiar finite-sample feature of quantile inference. Overall, these results accord with Theorems~\ref{the3}--\ref{the4}: the percentile bootstrap delivers accurate uncertainty quantification for DRM-based estimators over the considered designs.

\begin{table}[htbp]
\centering
\caption{Empirical coverage of percentile bootstrap CIs for $\bm{\theta}$}
\label{tab1}
\renewcommand{\arraystretch}{1.1}
\setlength{\tabcolsep}{5pt}

\begin{minipage}[t]{0.48\linewidth}
\centering
\subcaption*{Scenario I: Gamma ($d=3$)}
\begin{tabular}{lcccc}
\toprule
 & $F_1$ & $F_2$ & $F_3$ & $F_4$ \\
\midrule
$\theta_{r1}$ & .956 & .945 & .948 & .942 \\
$\theta_{r2}$ & .953 & .948 & .946 & .938 \\
$\theta_{r3}$ & .957 & .944 & .952 & .941 \\
\bottomrule
\end{tabular}
\end{minipage}
\hfill
\begin{minipage}[t]{0.48\linewidth}
\centering
\subcaption*{Scenario II: Normal ($d=3$)}
\begin{tabular}{lcccccc}
\toprule
 & $F_1$ & $F_2$ & $F_3$ & $F_4$ & $F_5$ & $F_6$ \\
\midrule
$\theta_{r1}$ & .959 & .952 & .945 & .941 & .943 & .944 \\
$\theta_{r2}$ & .957 & .953 & .943 & .948 & .950 & .949 \\
$\theta_{r3}$ & .955 & .951 & .942 & .947 & .954 & .949 \\
\bottomrule
\end{tabular}
\end{minipage}
\end{table}

\begin{table*}[h]
\centering
\caption{Empirical coverage of percentile bootstrap CIs for $F_r(x)$}
\label{tab2}
\renewcommand{\arraystretch}{1.1}
\setlength{\tabcolsep}{4pt}

\begin{minipage}[t]{0.48\textwidth}
\centering
\subcaption*{Scenario I: Gamma ($d=3$)}
\begin{tabular}{lccccc}
\toprule
 & $F_0$ & $F_1$ & $F_2$ & $F_3$ & $F_4$ \\
\midrule
$x= Q_0(0.1)$ & .938 & .941 & .916 & .917 & .944 \\
$x= Q_0(0.2)$ & .939 & .951 & .933 & .929 & .915 \\
$x= Q_0(0.3)$ & .949 & .948 & .932 & .934 & .927 \\
$x= Q_0(0.4)$ & .946 & .952 & .935 & .940 & .927 \\
$x= Q_0(0.5)$ & .950 & .954 & .940 & .945 & .936 \\
$x= Q_0(0.6)$ & .942 & .949 & .938 & .943 & .953 \\
$x= Q_0(0.7)$ & .941 & .946 & .944 & .947 & .948 \\
$x= Q_0(0.8)$ & .946 & .940 & .947 & .944 & .951 \\
$x= Q_0(0.9)$ & .943 & .936 & .951 & .953 & .947 \\
\bottomrule
\end{tabular}
\end{minipage}
\hfill
\begin{minipage}[t]{0.48\textwidth}
\centering
\subcaption*{Scenario II: Normal ($d=3$)}
\begin{tabular}{ccccccc}
\toprule
  $F_0$ & $F_1$ & $F_2$ & $F_3$ & $F_4$ & $F_5$ & $F_6$ \\
\midrule
 .926 & .935 & .938 & .943 & .932 & .933 & .922 \\
 .934 & .938 & .943 & .945 & .941 & .936 & .931 \\
 .936 & .945 & .942 & .948 & .944 & .943 & .939 \\
 .937 & .944 & .945 & .945 & .946 & .945 & .944 \\
 .942 & .945 & .951 & .951 & .949 & .952 & .950 \\
 .944 & .946 & .952 & .945 & .946 & .953 & .954 \\
 .941 & .949 & .950 & .948 & .947 & .954 & .953 \\
 .942 & .946 & .949 & .948 & .944 & .955 & .950 \\
 .940 & .942 & .945 & .946 & .944 & .958 & .954 \\
\bottomrule
\end{tabular}
\end{minipage}
\end{table*}
\begin{table*}[h]
\centering
\caption{Empirical coverage of percentile bootstrap CIs for $Q_r(p)$}
\label{tab3}
\renewcommand{\arraystretch}{1.1}
\setlength{\tabcolsep}{4pt}

\begin{minipage}[t]{0.48\textwidth}
\centering
\subcaption*{Scenario I: Gamma ($d=3$)}
\begin{tabular}{lccccc}
\toprule
 & $F_0$ & $F_1$ & $F_2$ & $F_3$ & $F_4$ \\
\midrule
$p=0.1$ & .925 & .923 & .921 & .929 & .945 \\
$p=0.2$ & .918 & .931 & .931 & .938 & .937 \\
$p=0.3$ & .935 & .933 & .934 & .930 & .937 \\
$p=0.4$ & .935 & .937 & .941 & .929 & .936 \\
$p=0.5$ & .936 & .937 & .938 & .933 & .924 \\
$p=0.6$ & .934 & .939 & .932 & .938 & .930 \\
$p=0.7$ & .933 & .939 & .935 & .933 & .922 \\
$p=0.8$ & .939 & .942 & .929 & .931 & .919 \\
$p=0.9$ & .939 & .928 & .941 & .933 & .906 \\
\bottomrule
\end{tabular}
\end{minipage}
\hfill
\begin{minipage}[t]{0.48\textwidth}
\centering
\subcaption*{Scenario II: Normal ($d=3$)}
\begin{tabular}{ccccccc}
\toprule
 $F_0$ & $F_1$ & $F_2$ & $F_3$ & $F_4$ & $F_5$ & $F_6$ \\
\midrule
  .904 & .924 & .934 & .941 & .941 & .944 & .945 \\
  .926 & .930 & .935 & .941 & .945 & .951 & .951 \\
  .918 & .930 & .939 & .943 & .939 & .949 & .946 \\
  .922 & .939 & .938 & .941 & .942 & .948 & .946 \\
  .928 & .941 & .938 & .932 & .935 & .946 & .942 \\
  .929 & .936 & .935 & .939 & .940 & .939 & .938 \\
  .937 & .935 & .938 & .936 & .933 & .937 & .940 \\
  .933 & .924 & .935 & .932 & .926 & .933 & .929 \\
  .935 & .933 & .933 & .924 & .934 & .926 & .914 \\
\bottomrule
\end{tabular}
\end{minipage}
\end{table*}

\clearpage
\section{Real-data analysis}
\label{se6}
We illustrate resampling-based inference under the DRM using per capita income data of
year 2020 from the China Family Panel Studies (CFPS)\footnotemark. The CFPS is a nationally representative, comprehensive longitudinal social survey designed to collect data at the community, family, and individual levels. Its objective is to reflect the social, economic, demographic, educational, and health transformations in China. 
The baseline survey covers 25 provinces/municipalities/autonomous regions, representing $95\%$ of the nation's population. Within the dataset, per capita income is derived as a composite variable, calculated by dividing the total household gross (or net) income by the household size. Household income itself is an aggregate measure comprising wage income, operating income, property income, transfer income, and other incomes, while household size is defined as the number of co-resident members.\footnotetext{Data source: China Family Panel Studies (CFPS, 2020). See \cite{CFPS2020} for details.}

We analyze data from six provinces: Henan ($n=1160$), Hubei ($n=163$), Hunan ($n=299$), Fujian ($n=166$), Anhui ($n=223$), and Sichuan ($n=510$). These provinces were chosen because they have comparable levels of economic development while also offering geographical diversity covering Central, East, and Southwest China which provides an ideal case for examining the general applicability of our method. Each province is treated as a distinct population.
Figure~\ref{fig:ecdf} and Table~\ref{tab:desc_stats} present the empirical distribution functions and the descriptive statistics for the six provinces, respectively. A key advantage of employing the DRM in this context is its ability to handle significant disparities in sample sizes across populations. As noted, our sample sizes range from $n=163$ for Hubei to $n=1160$ for Henan. A direct, non-parametric approach would yield estimates (e.g., for poverty rates or medians) with widely varying precision, making cross-province comparisons unreliable. The DRM framework mitigates this issue by jointly estimating the distributions, effectively ``borrowing strength" from the larger samples to improve the stability and precision of estimates for the smaller ones. By leveraging the information from all six provinces simultaneously under the model's structural assumption, we obtain more comparable and efficient inferences for each individual population.

\begin{figure}[htbp]
    \centering
    \includegraphics[width=0.6\linewidth]{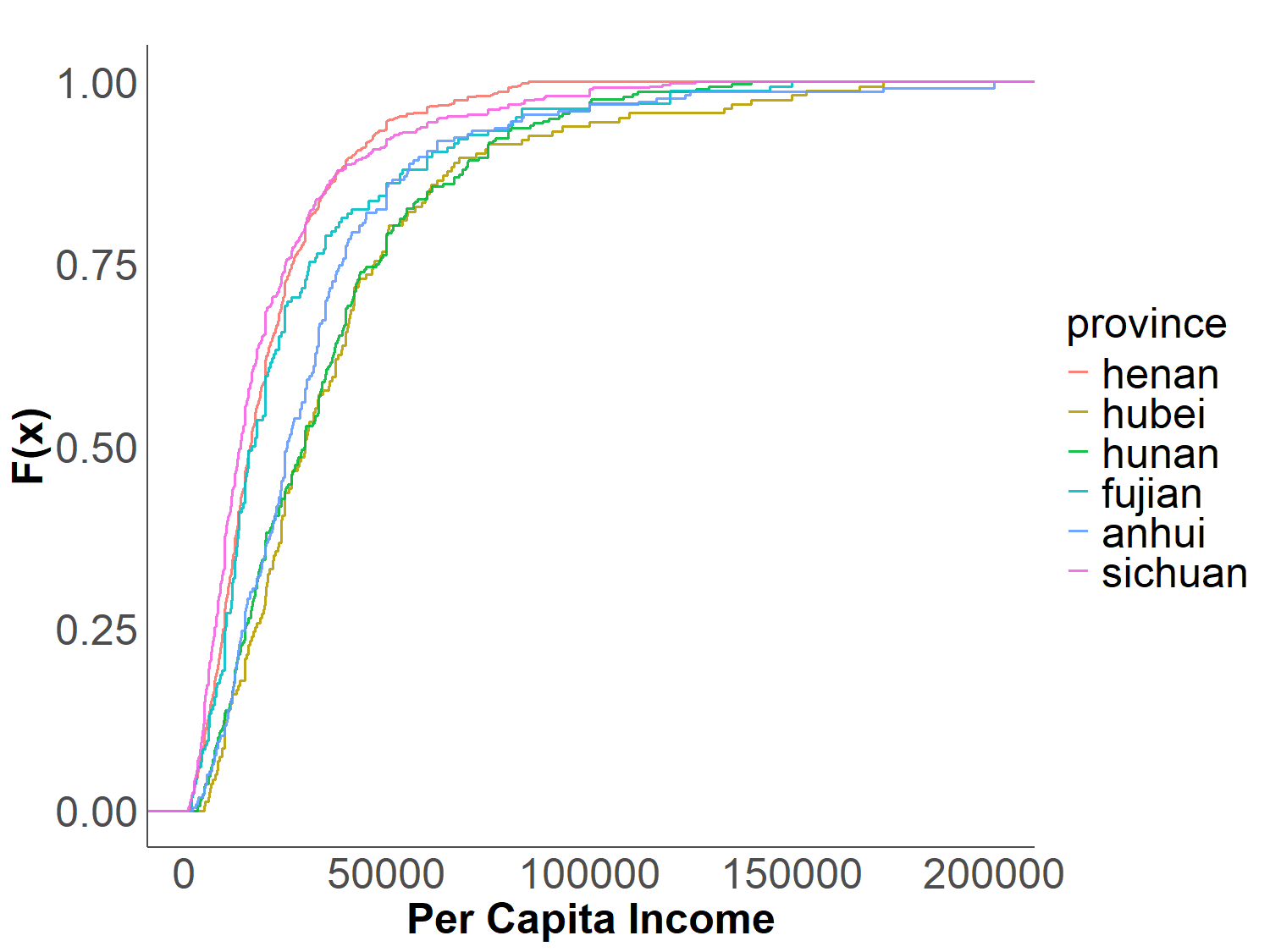}
    \caption{Empirical distribution functions for the six provinces}
    \label{fig:ecdf}
\end{figure}

\begin{table}[htbp]
\centering
\caption{Descriptive statistics  for the six provinces.}
\label{tab:desc_stats}
\begin{tabular}{@{}lrrrrrrr@{}}
\toprule
\textbf{Province} & \textbf{N} & \textbf{Mean} & \textbf{Median} & \textbf{StdDev} & \textbf{IQR} & \textbf{Min} & \textbf{Max} \\
\midrule
Henan   & 1160 & 20737 & 16362 & 16050 & 16794 & 1600 &  85000 \\
Hubei   &  163 & 38334 & 30000 & 32028 & 29150 & 5150 & 172500 \\
Hunan   &  299 & 35275 & 30000 & 26478 & 32550 & 3333 & 140000 \\
Fujian  &  166 & 26722 & 17083 & 27135 & 20575 & 2000 & 150000 \\
Anhui   &  223 & 32745 & 25483 & 28908 & 23767 & 2050 & 200000 \\
Sichuan &  510 & 20558 & 14000 & 20780 & 17419 & 1089 & 126300 \\
\bottomrule
\end{tabular}
\end{table}

We assume the six populations satisfy the DRM with the basis
$$\mathbf{q}(x)=\big(1,\ x,\ x^2,\ \log x,\ \sqrt{x}\big)^\T,$$
which is flexible enough to capture Gaussian- and Gamma-like shapes commonly observed in income distributions.
We fit the DRM to the six provinces jointly and obtain the implied fitted distribution $\widehat{F}_r(x)$ for each province $r$. For visualization, Figure~\ref{fig:density_overlay} displays, for each province, a histogram of income (x-axis labeled in thousands of yuan) overlaid with the \emph{DRM-implied} kernel density (convolution of $\widehat{F}_r$ with a Gaussian kernel). The DRM density estimates clearly match the histograms well, 
supporting the model assumption. 

\begin{figure}[htbp]
    \centering
    \begin{subfigure}[b]{0.48\textwidth}
        \centering
        \includegraphics[width=\textwidth]{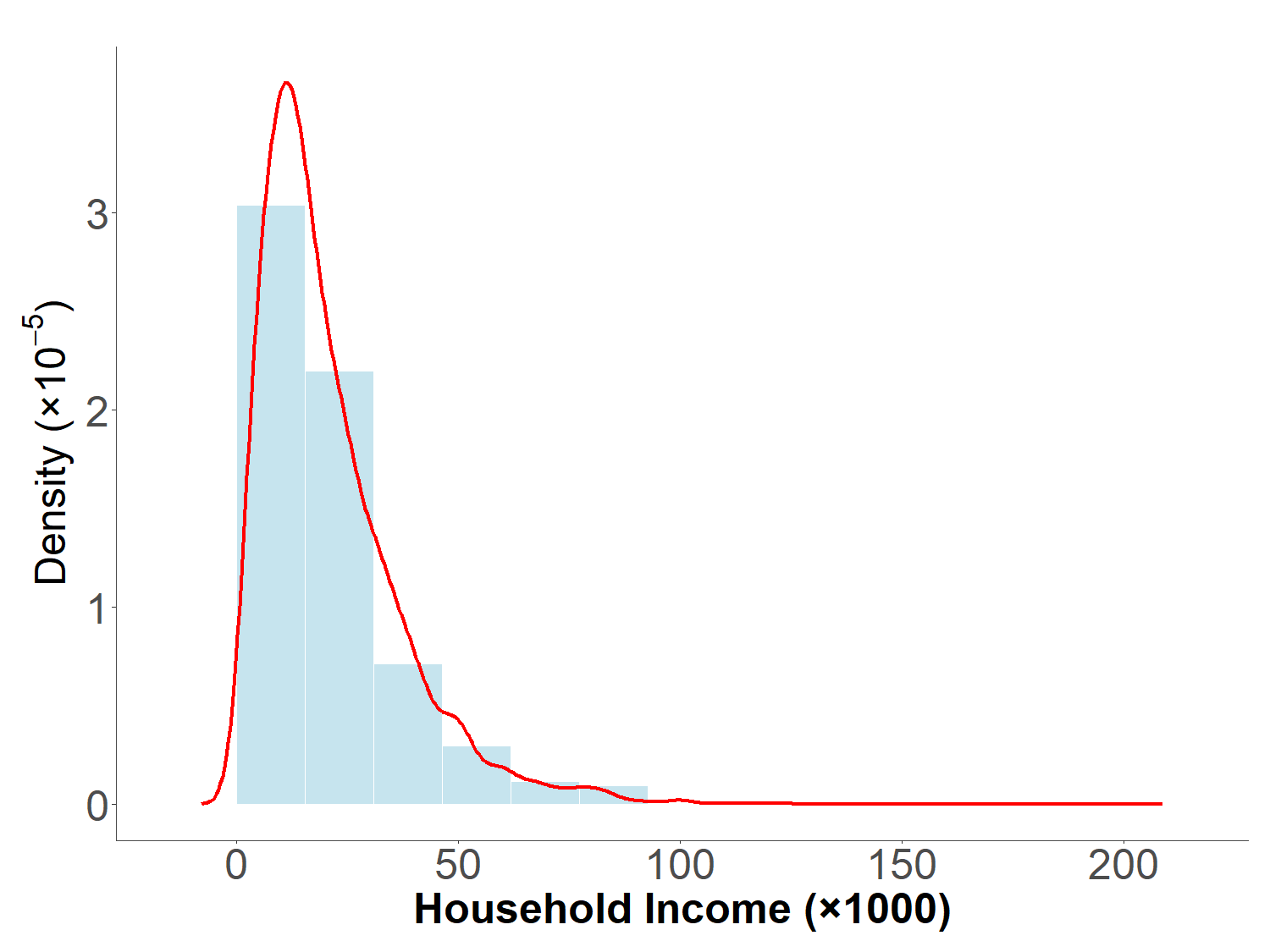}
        \subcaption{Henan}
        \label{fig:sub_henan}
    \end{subfigure}
    \hfill 
    \begin{subfigure}[b]{0.48\textwidth}
        \centering
        \includegraphics[width=\textwidth]{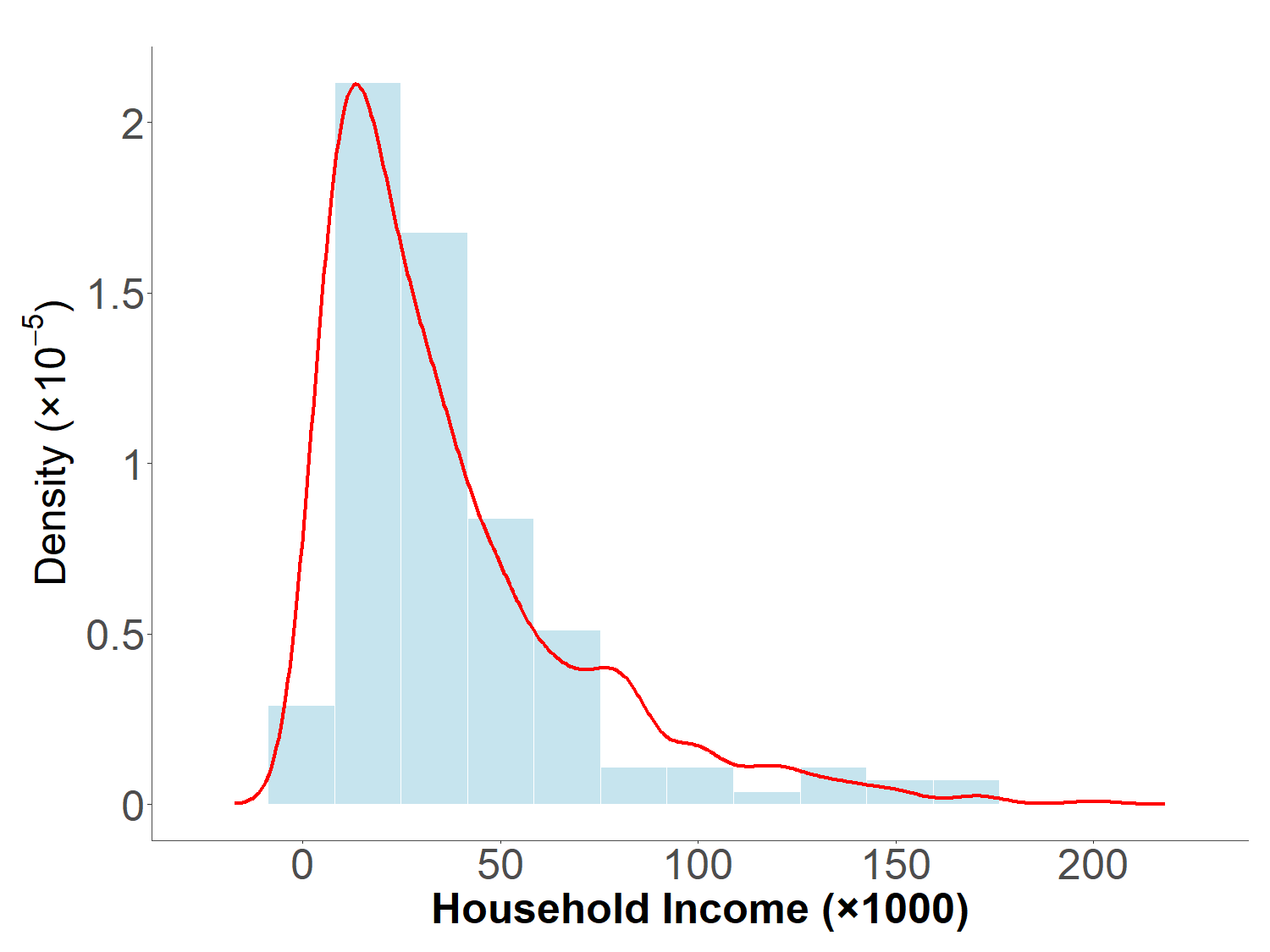}
        \subcaption{Hubei}
        \label{fig:sub_hubei}
    \end{subfigure}

    \vspace{0.5cm}

    \begin{subfigure}[b]{0.48\textwidth}
        \centering
        \includegraphics[width=\textwidth]{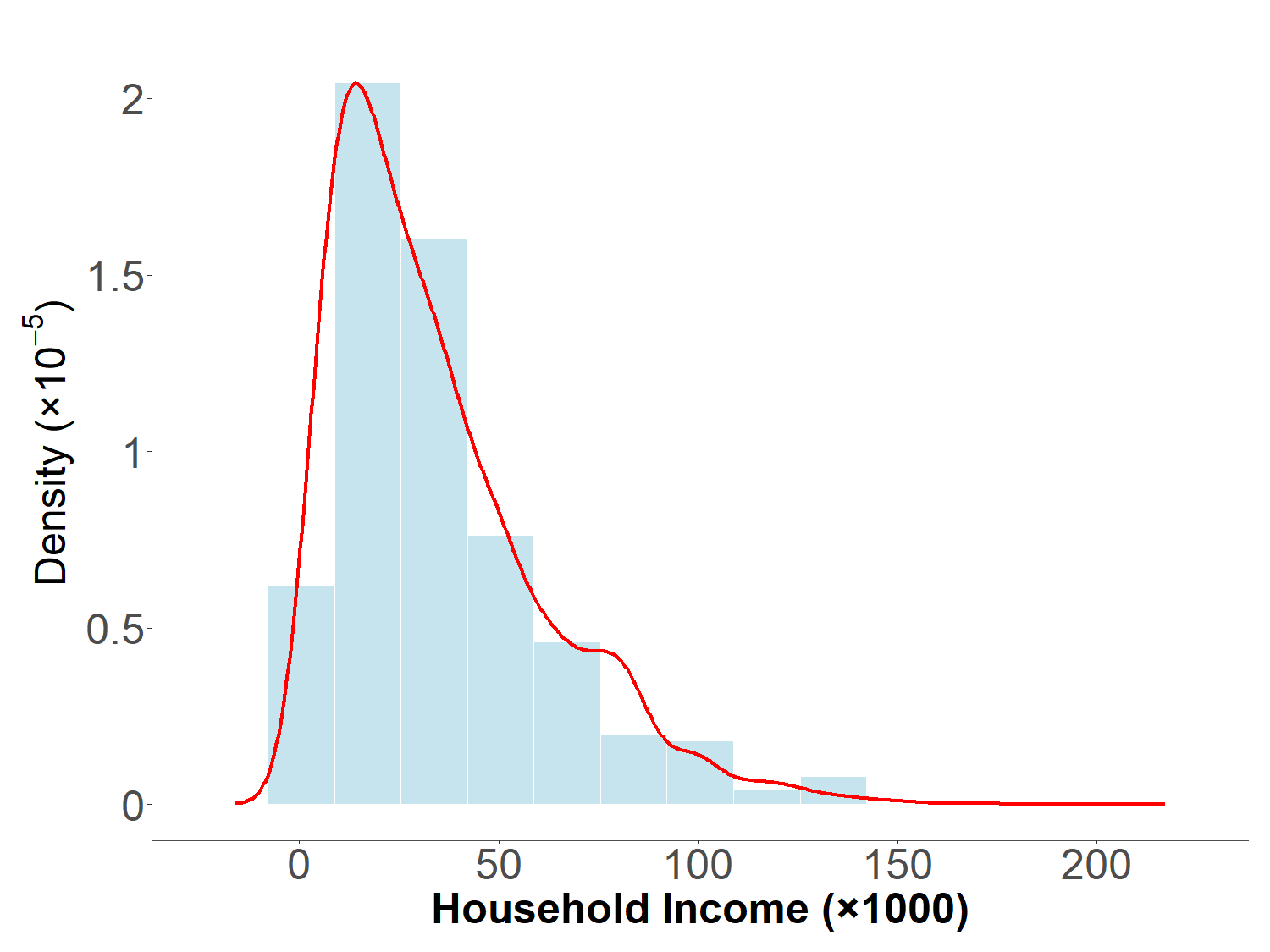}
        \subcaption{Hunan}
        \label{fig:sub_hunan}
    \end{subfigure}
    \hfill 
    \begin{subfigure}[b]{0.48\textwidth}
        \centering
        \includegraphics[width=\textwidth]{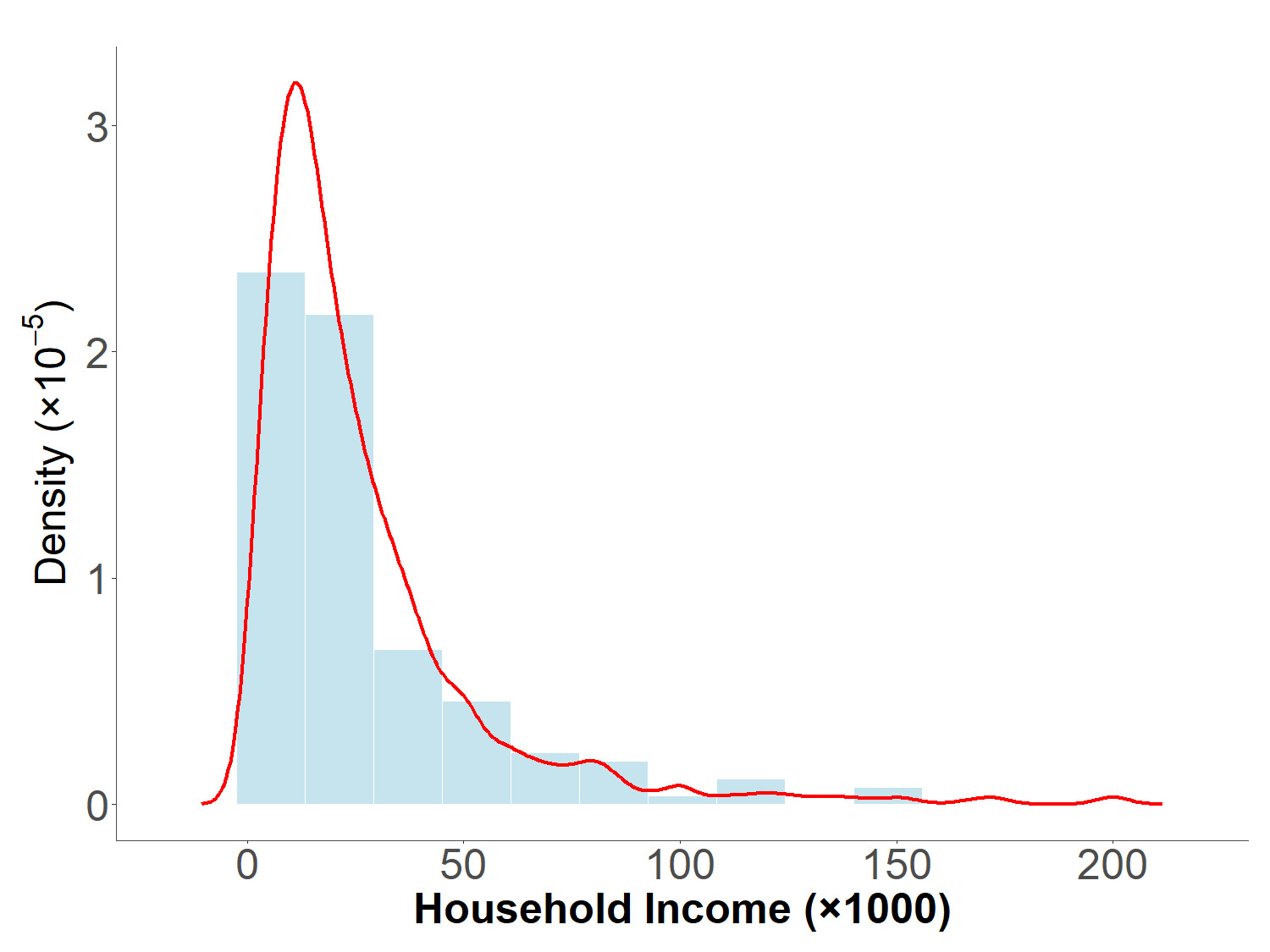}
        \subcaption{Fujian}
        \label{fig:sub_fujian}
    \end{subfigure}
     \vspace{0.5cm}

    \begin{subfigure}[b]{0.48\textwidth}
        \centering
        \includegraphics[width=\textwidth]{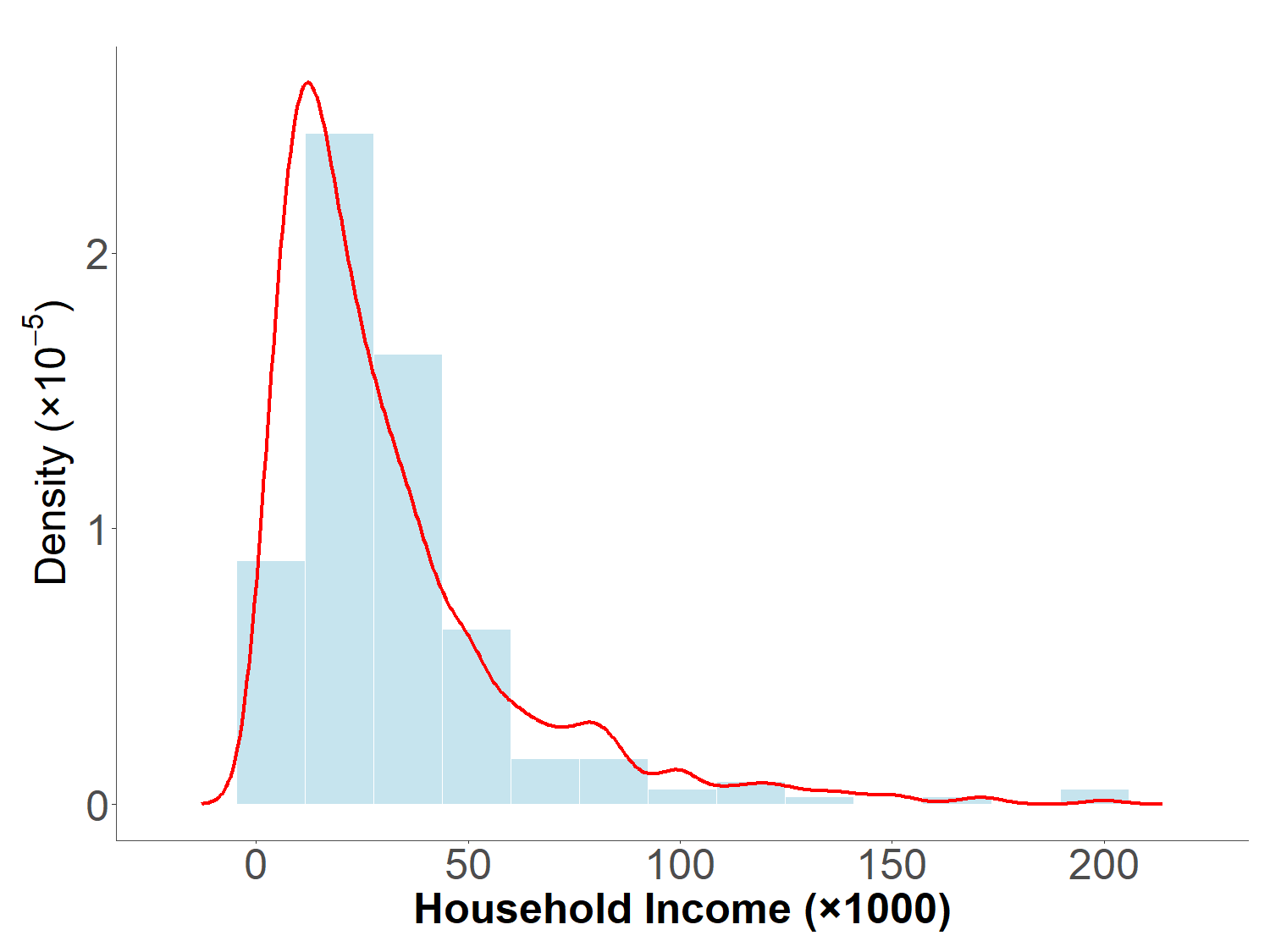}
        \subcaption{Anhui}
        \label{fig:sub_hunan}
    \end{subfigure}
    \hfill 
    \begin{subfigure}[b]{0.48\textwidth}
        \centering
        \includegraphics[width=\textwidth]{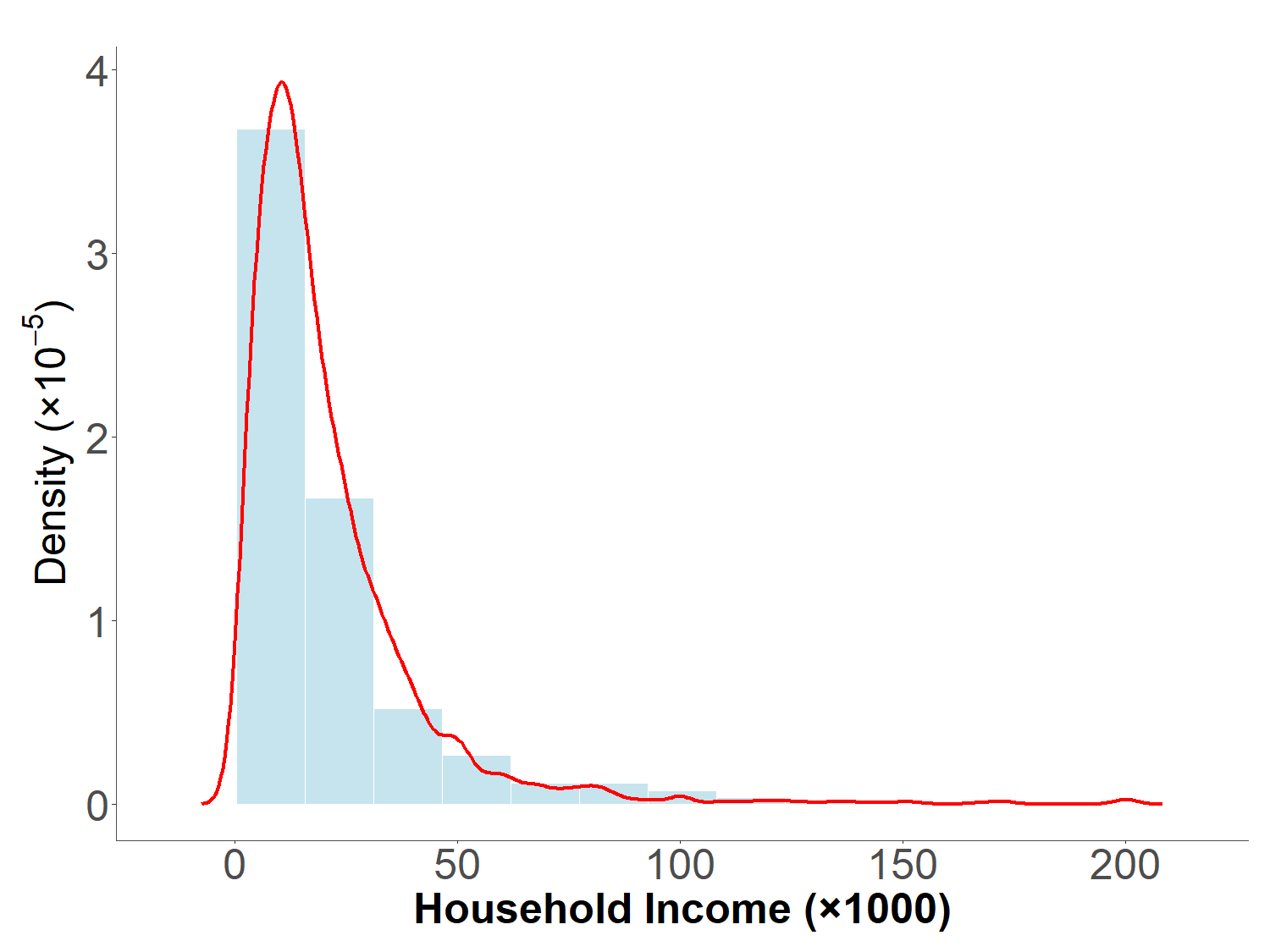}
        \subcaption{Sichuan}
        \label{fig:sub_fujian}
    \end{subfigure}
    \caption{Income histograms overlaid with DRM-implied density estimates.}
    \label{fig:density_overlay}
\end{figure}

Let us examine the poverty rates, median incomes, and distributional dominance among these six provinces,
which can provide vital information to policymakers. 
We apply the bootstrap method to construct their $95\%$ two-sided confidence intervals. 
More specifically, the parameters of interests are:
\begin{itemize}\setlength{\itemsep}{2pt}
\item \textbf{Poverty rate.} 
The poverty line was set at an annual income of 4,000 RMB at year 2020.
Hence, the poverty rate is $F(4000)$.

\item \textbf{Median income.} 
This parameter is plainly $Q(0.5)$.

\item \textbf{Dominance index.} 
The dominance index of province $r$ over province $s$ is given by
\[
\gamma(F_r, F_s)\;=\;\mu\{\,p\in(0,1): Q_r(p)>Q_s(p)\,\}.
\]
A value above $1/2$ reflects more individuals in province $r$
has improved income over province $s$.
\end{itemize}
We work out the percentile bootstrap $95\%$ confidence intervals (CIs) as in \eqref{boot.CI}.
We use $B = 20001$ bootstrap replicates. 
This procedure respects the joint estimation structure of the DRM and 
is supported by Theorems~\ref{the3}--\ref{the4}.

Figure~\ref{fig:forest_panels} presents two ``forest-style'' panels:
\begin{enumerate}\setlength{\itemsep}{2pt}
\item[(a)] Poverty rates for these six provinces with $95\%$ CIs.
\item[(b)] Median incomes with $95\%$ CIs.
\end{enumerate}
In each panel, the point marks the DRM estimate 
and the horizontal bar the bootstrap CI. 

\begin{figure}[ht]
    \centering
    \begin{subfigure}[b]{0.48\textwidth}
        \centering
        \includegraphics[width=\textwidth]{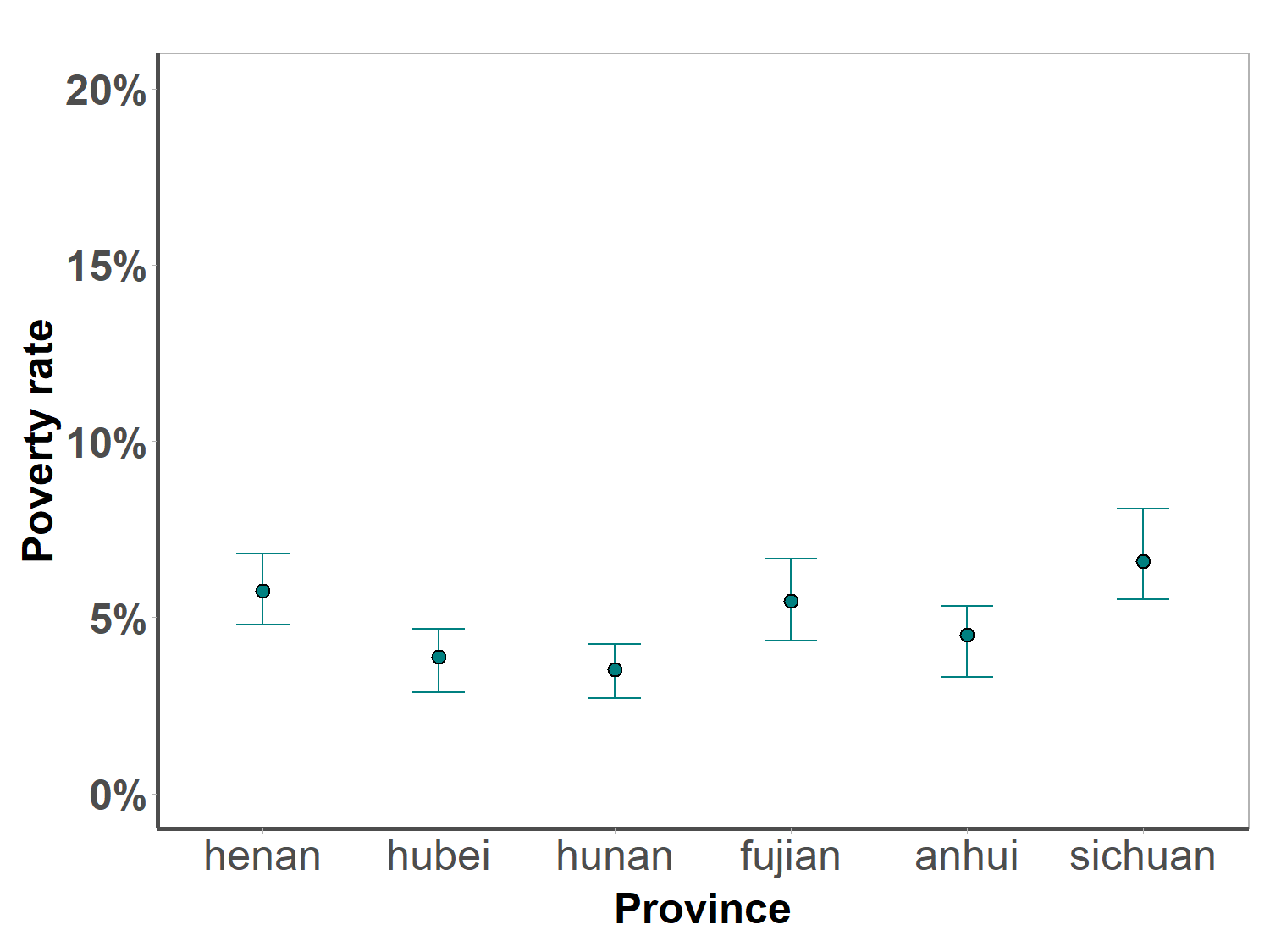}
        \subcaption{}
        \label{fig:poverty_rates}
    \end{subfigure}
    \hfill %
    \begin{subfigure}[b]{0.48\textwidth}
        \centering
        \includegraphics[width=\textwidth]{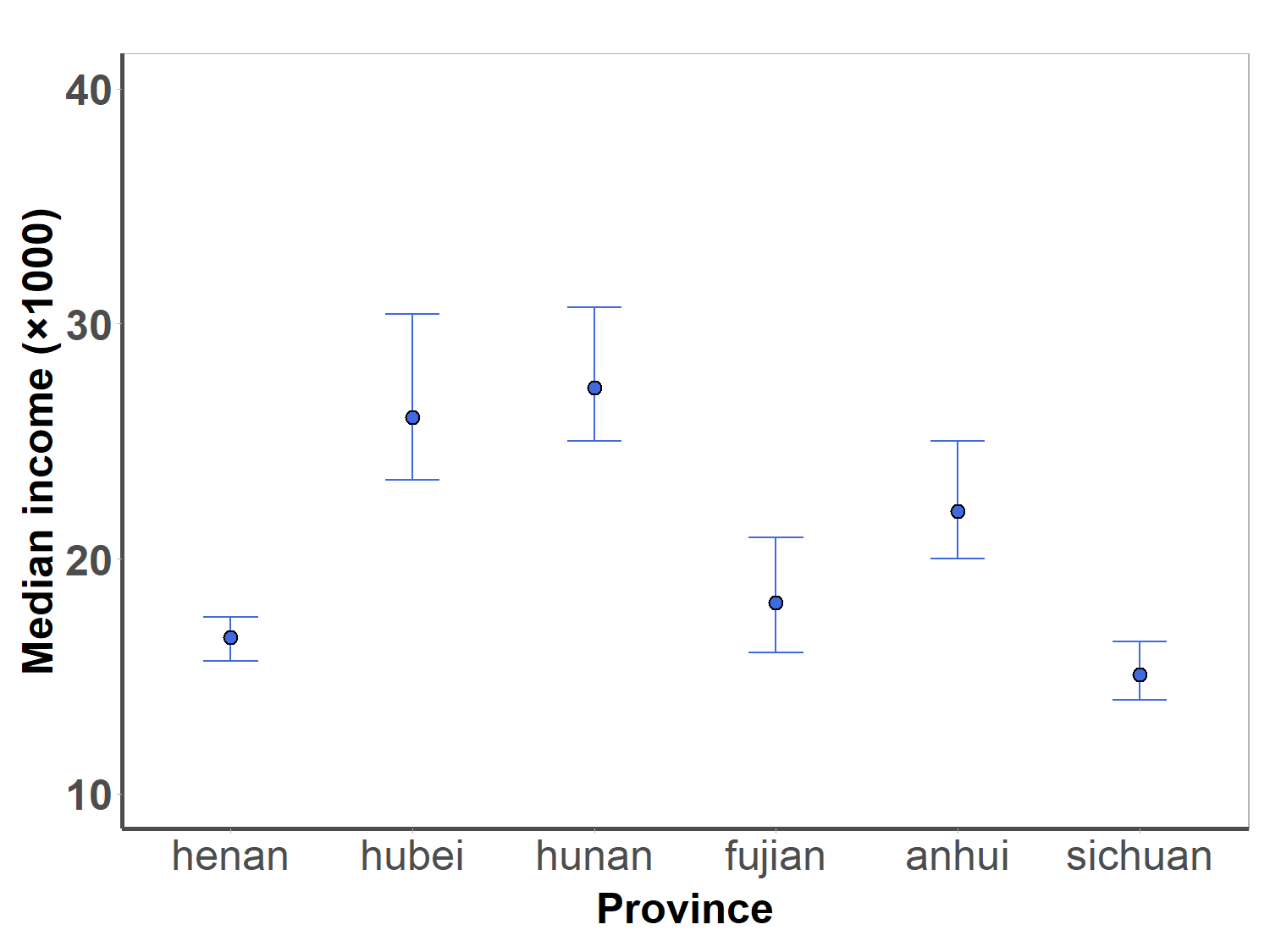}
        \subcaption{}
        \label{fig:median_incomes}
    \end{subfigure}

    \caption{DRM estimates and bootstrap $95\%$ CIs for various indicators.}
    \label{fig:forest_panels}
\end{figure}

Table~\ref{tab:gamma_indices} regards dominance indices,
their DRM based point estimates and bootstrap CIs.
\begin{table}[ht] 
\centering
\caption{Pairwise Dominant Indices $\gamma(F, G)$}
\label{tab:gamma_indices}
\resizebox{\textwidth}{!}{
\begin{tabular}{@{}llccccc@{}}
\toprule
\multicolumn{2}{c}{\multirow{2}{*}{\textbf{Population F (Row)}}} & \multicolumn{5}{c}{\textbf{Population G (Column)}} \\
\cmidrule(l){3-7}
 && \textbf{Hubei} & \textbf{Hunan} & \textbf{Fujian} & \textbf{Anhui} & \textbf{Sichuan} \\ 
\midrule
\multirow{2}{*}{\textbf{Henan} }  & Estimate &  0.000 & 0.000 & 0.000 & 0.000 & 0.888 \\
                            & CI         & (0.000, 0.000) & (0.000, 0.000) & (0.000, 0.651) & (0.000, 0.000) & (0.289, 0.969) \\
\addlinespace 

\multirow{2}{*}{\textbf{Hubei} }  & Estimate &  & 0.290 & 0.991 & 0.995 & 0.998 \\
                            & CI        & & (0.007, 0.988) & (0.930, 1.000) & (0.422, 0.998) & (0.989, 1.000) \\
\addlinespace

\multirow{2}{*}{\textbf{Hunan }}  & Estimate &  &   & 0.969 & 0.926 & 0.993 \\
                            & CI       &    &  & (0.900, 0.997) & (0.724, 0.996) & (0.981, 0.998) \\
\addlinespace

\multirow{2}{*}{\textbf{Fujian}  } & Estimate &  &    &  & 0.011 & 0.997 \\
                            & CI       &    &  &  & (0.000, 0.199) & (0.930, 0.999) \\
\addlinespace

\multirow{2}{*}{\textbf{Anhui}}  & Estimate &  &    &  &  & 0.998 \\
                            & CI       &    &  &  &  & (0.978, 1.000) \\
\bottomrule

\end{tabular}
}
\begin{tablenotes}
  \item[\footnotesize Note:] \footnotesize Each cell contains the point estimate for $\gamma(F, G)$ (row F, column G) with its corresponding confidence interval (CI) in parentheses.
\end{tablenotes}
\end{table}

The analysis reveals significant economic heterogeneity among the six provinces. 
Two panels in Figure~\ref{fig:forest_panels} paint clear pictures of income levels 
and their associated uncertainty, highlights this disparity between provinces. 
Specifically, Figure~\ref{fig:median_incomes} shows a clear stratification of median incomes: 
Hubei and Hunan are notable leaders, with estimated medians reaching 30,000 RMB, 
while Sichuan (approx. 14,000 RMB) and Henan (16,362 RMB) lag significantly, 
placing them at the lower end, with Fujian and Anhui occupying intermediate positions. 
Concurrently, Figure~\ref{fig:poverty_rates} shows that the 2020 poverty rates
are generally below $8\%$, with Hubei and Hunan fair much better. 
The dominance index, $\gamma(F, G)$, in Table~\ref{tab:gamma_indices}
provides a comprehensive comparison of the entire income distributions. 
The results clearly reveal strong dominance relationships; 
for example, the estimated $\gamma(\text{Hubei}, \text{Henan})$ is 0.998 with a CI of (0.989, 1.000), implying Hubei's per capita income distribution practically dominants that of Henan. 
Also, Sichuan is practically dominated by all other provinces except for Henan at a less strength.
The indices also uncover complex relationships and uncertainty. 
An interesting finding is $\gamma(\text{Fujian}, \text{Anhui})$, 
with a point estimate of 0.011 and CI (0.000, 0.199), 
suggests that, contrary to what a simple median comparison might imply, 
more residents in Anhui are fair better than Fujian. 
With a confidence interval ($0.000, 0.651$) for $\gamma(\text{Henan}, \text{Fujian})$
which contains the middle ground $0.5$, the income levels of residents
in these two provinces are indifferent statistically.

\section{Conclusions}
\label{se7}
We provide a unified theory for bootstrap inference under the DRM. First, we establish distributional equivalence between DRM estimators and their bootstrap counterparts for model parameters. Second, we extend convergence for DRM distribution estimators from pointwise to process level, enabling asymptotics for functionals such as quantiles. Third, we show the corresponding bootstrap process shares the same weak limit. Together with simulations and a real-data application, these results offer a rigorous foundation for the widespread use of bootstrap methods in DRM analysis.

\section*{Acknowledgment}
This work was supported by National Natural Science Foundation of China (No. 72571262).
\bibliographystyle{biometrika}
\bibliography{paper-ref}
\end{document}